\documentclass[11pt,oneside,english]{amsart}
\usepackage[T1]{fontenc}
\usepackage[latin9]{inputenc}
\usepackage{geometry}
\usepackage{color}
\usepackage{babel}
\usepackage{wrapfig}
\usepackage{units}
\usepackage{amsthm}
\usepackage{amstext}
\usepackage{amssymb}
\usepackage{graphicx}
\usepackage{xargs}[2008/03/08]
\usepackage[unicode=true,pdfusetitle,
 bookmarks=true,bookmarksnumbered=false,bookmarksopen=false,
 breaklinks=false,pdfborder={0 0 1},backref=false,colorlinks=true]
 {hyperref}
\usepackage{breakurl}

\makeatletter

\newcommand{\noun}[1]{\textsc{#1}}
\providecommand{\tabularnewline}{\\}

\numberwithin{equation}{section}
\numberwithin{figure}{section}
  \theoremstyle{remark}
  \newtheorem*{rem*}{\protect\remarkname}
\theoremstyle{plain}
\newtheorem{thm}{\protect\theoremname}[section]
  \theoremstyle{remark}
  \newtheorem{rem}[thm]{\protect\remarkname}
  \theoremstyle{plain}
  \newtheorem*{thm*}{\protect\theoremname}
 \newcommand\thmsname{\protect\theoremname}
 \newcommand\nm@thmtype{theorem}
 \theoremstyle{plain}
 
 \newenvironment{namedthm}[1][Undefined Theorem Name]{
   \ifx{#1}{Undefined Theorem Name}\renewcommand\nm@thmtype{theorem*}
   \else\renewcommand\thmsname{#1}\renewcommand\nm@thmtype{namedtheorem}
   \fi
   \begin{\nm@thmtype}}
   {\end{\nm@thmtype}}
  \theoremstyle{definition}
  \newtheorem{defn}[thm]{\protect\definitionname}
  \theoremstyle{plain}
  \newtheorem{prop}[thm]{\protect\propositionname}
  \theoremstyle{plain}
  \newtheorem{lem}[thm]{\protect\lemmaname}
  \theoremstyle{plain}
  \newtheorem{cor}[thm]{\protect\corollaryname}

\usepackage{kpfonts}
\usepackage{esint}

\makeatother

  \providecommand{\corollaryname}{Corollary}
  \providecommand{\definitionname}{Definition}
  \providecommand{\lemmaname}{Lemma}
  \providecommand{\propositionname}{Proposition}
  \providecommand{\remarkname}{Remark}
  \providecommand{\theoremname}{Theorem}
\providecommand{\theoremname}{Theorem}

\begin{document}
\global\long\def\tx#1{\textrm{#1}}
\global\long\def\ww#1{\mathbb{#1}}
\global\long\def\pp#1{\frac{\partial}{\partial#1}}
\global\long\def\dd#1{\tx d#1}
\global\long\def\germ#1{\ww C\left\{  #1\right\}  }
\global\long\def\pol#1#2{\ww C\left[#1\right]_{#2}}
\newcommandx\polg[1][usedefault, addprefix=\global, 1=y]{\frak{N}\left\{  x,#1\right\}  }
\global\long\def\flow#1#2#3{\Phi_{#1}^{#2}#3}
\global\long\def\per#1{\mathcal{T}_{#1}}
\global\long\def\sec#1{\mathcal{S}_{#1}}
\global\long\def\nf#1#2{\nicefrac{#1}{#2}}
\global\long\def\zsk{\nf{\ww Z}{k\ww Z}}
\global\long\def\norm#1#2{\left|\left|#1\right|\right|_{#2}}
\global\long\def\re#1{\Re\left(#1\right)}
\global\long\def\im#1{\Im\left(#1\right)}
\global\long\def\frml#1{\ww C\left[\left[#1\right]\right]}
\global\long\def\hol#1{\mathfrak{h}_{#1}}
\global\long\def\id{\tx{Id}}
\global\long\def\ii{\tx i}
\global\long\def\hot{\tx{h.o.t.}}

\title{Analytic normal forms for convergent saddle-node vector fields{*}}

\author{Reinhard SCHÄFKE, Loïc TEYSSIER}

\thanks{The work of the first author was supported in part by grants of the
French National Research Agency (ref. ANR-10-BLAN 0102 and ANR-11-BS01-0009).}

\date{July 2013\\
{*}Preprint}

\address{Laboratoire I.R.M.A. (Université de Strasbourg \& C.N.R.S.)}

\email{\texttt{schaefke@math.unistra.fr} --\&-- \texttt{teyssier@math.unistra.fr}}
\begin{abstract}
We give unique analytic <<normal forms>> for germs of a holomorphic
vector field of the complex plane in the neighborhood of an isolated
singularity of saddle-node type having a convergent formal separatrix.
We specifically address the problem of computing the normal form explicitly.
\end{abstract}
\maketitle

\section{Introduction and statement of the results}

The general question of finding the <<simplest>> form of a dynamical
system through changes preserving its qualitative properties is central
in the theory. A simpler form often means a better understanding of
the behavior of the system. This article is concerned with finding
<<simple>> models of holomorphic dynamical system given by the flow
of a vector field near a singularity of convergent saddle-node kind%
\footnote{All relevant definitions will be given in due time in the course of
the introduction.%
} in $\ww C^{2}$. We use the technical term <<normal form>> for
vector fields brought into these forms, although the latter do not
satisfy the properties usually required in the normal form theory.
While this imprecision in the language may be confusing its usage
is nonetheless becoming more and more spread to refer to <<simple>>
forms which are essentially unique (say, up to the action of a finite
dimensional space).

It is possible to attach to a vector field%
\footnote{As usual we identify vector fields in $\ww C^{2}$ given as vector
valued functions $\left[\begin{array}{c}
A\\
B
\end{array}\right]$ with the Lie (directional) derivative $A\pp x+B\pp y$ acting on
functions or power series $F$ by
\begin{eqnarray*}
\left(A\pp x+B\pp y\right)\cdot F & := & A\frac{\partial F}{\partial x}+B\frac{\partial F}{\partial y}\,.
\end{eqnarray*}
} $Z:=A\pp x+B\pp y$ two dynamical systems: the one induced by the
flow itself, and the one related to the underlying foliation. The
objects of study in the former case are the integral curves of $Z$
and their natural parametrization by the flow, that is the solutions
to the differential system
\begin{eqnarray*}
\begin{cases}
\dot{x}\left(t\right) & =A\left(x\left(t\right),y\left(t\right)\right)\\
\dot{y}\left(t\right) & =B\left(x\left(t\right),y\left(t\right)\right)
\end{cases} &  & ,
\end{eqnarray*}
 whereas in the latter case only their images are of interest: the
leaves of the foliation $\mathcal{F}_{Z}$ are the images of the integral
curves regardless of how they are parametrized. They correspond to
the graphs of the solutions of the ordinary differential equation
\begin{eqnarray*}
A\left(x,y\left(x\right)\right)y'\left(x\right) & = & B\left(x,y\left(x\right)\right)\,.
\end{eqnarray*}
 Therefore two vector fields induce the same foliation when they differ
by multiplication with a non-vanishing function. 

\bigskip{}

Being given a (germ of a) holomorphic vector field we want to simplify
its components using local analytic changes of coordinates. In a first
step, one tries to simplify the vector field as much as possible using
formal power series. In the case of saddle-node vector fields this
process leads to polynomial formal normal forms~\cite{Bruno,Dulac}.
Yet this formal approach does not preserve the dynamics. Nevertheless
analyzing the divergence of these formal transforms provides a lot
of information about the dynamics or the solvability (in some differential
field) of the system. The theory of summability was successfully used
to perform this task\noun{~}\cite{MaRa-SN}, yielding a complete
set of functional invariants to classify the foliation. However their
construction did not directly yield normal forms. Some years later
the complete modulus of classification for saddle-node vector fields
was given in~\cite{MeshVoro,Tey-SN} by appending to the orbital
modulus another functional invariant. Still no normal form was proposed.

In~\cite{Tey-ExSN}, a <<first order>> form allowed one to decide
in some cases whether two vector field are (orbitally or not) conjugate,
but the given form was far from being unique. At the same time \noun{F.~Loray~}\cite{Loray}
performed a geometric construction providing normal forms in some
cases, generalizing the ones stated by \noun{J.~Écalle} in~\cite{Ecal}.
In this article we present a generalization of the approach of~\cite{Tey-ExSN}
and provide normal forms in every case. Moreover, we prove that this
method is constructive or, at least, computable numerically and in
some cases symbolically. These results are related to recent works
of \noun{O.~Bouillot}~\cite{BouillonThese,Bouillon}.

\subsection{Statement of the results}

~

\begin{wrapfigure}{o}{4.5cm}%
\includegraphics[width=4cm]{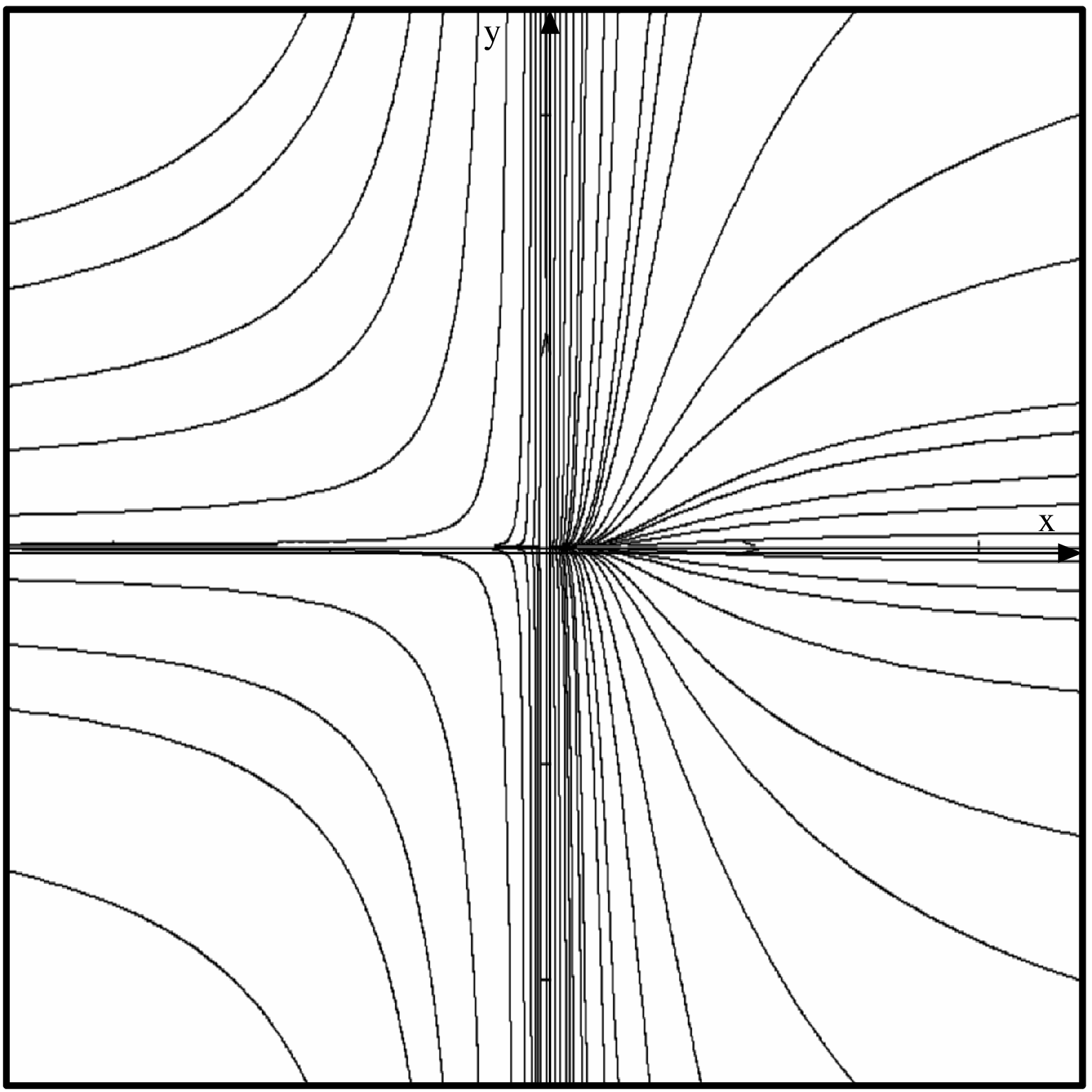}

Real trajectories of \\
$x^{2}\pp x+y\pp y$\end{wrapfigure}%
 A \textbf{saddle-node} vector field is a germ of a holomorphic vector
field $Z$ near some isolated singularity, which we conveniently locate
at $\left(0,0\right)$, such that the linear part has two eigenvalues,
exactly one of which is nonzero. In other words we require that $A\left(0,0\right)=B\left(0,0\right)=0$
is locally the only common zero of the components, and that the spectrum
of the matrix $\left[\nabla A\left(0,0\right),\nabla B\left(0,0\right)\right]$
is $\left\{ 0,\lambda_{2}\right\} $ with $\lambda_{2}\neq0$. Generically
there exists only one leaf with analytic closure at $\left(0,0\right)$,
tangent to the eigenspace of $\lambda_{2}$, and a formal one tangent
to the other eigenspace. When this formal separatrix is a convergent
power series we say that $Z$ is \textbf{convergent}, and divergent
otherwise. 

Regardless of the convergence or not of the formal separatrix, the
vector field $Z$ is always formally conjugate to one of the formal
normal form 
\begin{eqnarray*}
P\left(x\right)\left(x^{k+1}\pp x+y\left(1+\mu x^{k}\right)\pp y\right)
\end{eqnarray*}
where $k\in\ww N$ is a positive integer, $\mu\in\ww C$ and $P$
is a polynomial of degree at most $k$ with $P\left(0\right)=\lambda_{2}$.
\emph{Throughout the article, we fix a positive integer $k$}. This
form is unique up to linear changes of variables $x\mapsto\alpha x$
with $\alpha^{k}=1$, acting on $P$ by right composition. The complex
number $\mu$ is the \textbf{formal orbital modulus~}\cite{Dulac}
while (the class of) $P$ is the \textbf{formal temporal modulus~}\cite{Bruno}. 
\begin{rem*}
Recall that two vector fields $Z$ and $\tilde{Z}$ are called (formally,
analytically) \textbf{conjugate} when there exists a (formal, analytic)
change of variables $\Psi$ such that $\Psi^{*}Z=\tilde{Z}$, while
they are (formally, analytically) \textbf{orbitally conjugate} when
$\Psi^{*}Z$ induces the same foliation as $\tilde{Z}$, which we
write $\Psi^{*}\mathcal{F}_{Z}=\mathcal{F}_{\tilde{Z}}$. In order
to determine the vector field $\tilde{Z}$ obtained by changing the
coordinates in $Z$ by some transformation $\Psi$, \emph{i.e.} $\tilde{Z}=\Psi^{*}Z$,
one can use the relation
\begin{eqnarray*}
\tilde{Z}\cdot\Psi & = & Z\circ\Psi\,,
\end{eqnarray*}
where $\tilde{Z}\cdot$ denotes the Lie derivative%
\footnote{The Lie derivative acts component-wise on vectors.%
} associated to $\tilde{Z}$. 
\end{rem*}
The complete orbital analytical classification of saddle-node vector
fields is due to \noun{J.~Martinet} and \noun{J.-P.~Ramis~}\cite{MaRa-SN}.
In the case $k=1$ the temporal classification was obtained by \noun{Y.~Mershcheryakova}
and \noun{S.~Voronin}~\cite{MeshVoro} and at the same time by \noun{L.~Teyssier~}\cite{Tey-SN}
in the general case. The corresponding \textbf{classification modulus}
of a convergent vector $Z$ is a $\left(2+2k\right)$\textendash{}tuple
\begin{eqnarray*}
\mathcal{M}\left(Z\right) & := & \left(\mu,P\right)\oplus\left(\varphi^{j},f^{j}\right)_{j\in\zsk}
\end{eqnarray*}
where $\varphi^{j},\, f^{j}\in h\germ h$. The pair $\left(\mu,P\right)$
is the formal modulus, as explained above. The way how $Z$ relates
to both its \textbf{orbital modulus} $\left(\varphi^{j}\right)_{j\in\zsk}$
and its \textbf{temporal modulus} $\left(f^{j}\right)_{j\in\zsk}$
is described further down. 

Any element of
\begin{eqnarray*}
{\tt Mod}_{k} & := & \ww C\times\pol x{\leq k}\oplus\left(h\germ h\right)^{2k}\\
 & = & \left\{ \left(\mu,P\right)\oplus\left(\varphi^{j},f^{j}\right)_{j\in\zsk}\right\} 
\end{eqnarray*}
is the modulus of some saddle-node vector field. In this paper we
are mainly concerned with giving a «constructive» proof of this property.
We intend to generalize our approach to divergent saddle-node vector
fields in an upcoming work dealing with the analogous problem for
resonant-saddles.
\begin{rem}
\label{rem_eqv_mod}There exists a natural action of 
\begin{eqnarray*}
{\tt Aut}_{k} & := & \zsk\times\ww C_{\neq0}\\
 & = & \left\{ \left(\theta,c\right)\right\} 
\end{eqnarray*}
 on ${\tt Mod}_{k}$ by cyclic permutation $j\mapsto j+\theta$ of
the indexes and right-composition by the linear maps $\left(x,h\right)\mapsto\left(e^{\nf{2\ii\pi\theta}k}x,ch\right)$.
The actual moduli space is the quotient $\nf{{\tt Mod}_{k}}{{\tt Aut}_{k}}$,
in the sense that $Z$ and $\tilde{Z}$ are locally analytically conjugate
if, and only if, $\mathcal{M}\left(Z\right)$ and $\mathcal{M}\left(\tilde{Z}\right)$
are conjugate by the action of ${\tt Aut}_{k}$.
\end{rem}
The starting point of this article is the following result due to
\noun{F.~Loray}:
\begin{thm*}
\cite{Loray} Any convergent saddle-node vector field with formal
modulus $k=1$ is orbitally conjugate to a vector field of the form
\begin{eqnarray*}
x^{2}\pp x+y\left(1+\mu x+xR\left(x^{\sigma}y\right)\right)\pp y
\end{eqnarray*}
where $R$ is a germ of a holomorphic function vanishing at the origin,
and $\sigma$ is defined by
\begin{eqnarray*}
\begin{cases}
\sigma:=0 & \mbox{if }\mu\notin\ww R_{\leq0}\\
\sigma:=\left\lfloor -\mu\right\rfloor +1 & \mbox{otherwise}
\end{cases} &  & .
\end{eqnarray*}
The germ $R$ is unique up to the action of $\ww C_{\neq0}$ through
linear changes of coordinates $y\mapsto cy$.
\end{thm*}
We present a generalization of this result to the non-generic case
$k>1$, which provides orbital normal forms as well as normal forms
for vector fields.
\begin{namedthm}
[Main Theorem]\label{thm:convergent}Let $Z$ be a germ of a convergent
saddle-node vector field. Then $Z$ is analytically conjugate to a
vector field of the form:
\begin{eqnarray*}
\frac{P\left(x\right)}{1+xP\left(x\right)G\left(x,x^{\sigma}y\right)}\left(x^{k+1}\pp x+y\left(1+\mu x^{k}+xR\left(x,x^{\sigma}y\right)\right)\pp y\right)
\end{eqnarray*}
where $G\left(x,u\right)=\sum_{n>0}G_{n}\left(x\right)u^{n}$ and
$R\left(x,u\right)=\sum_{n>0}R_{n}\left(x\right)u^{n}$ are germs
of a holomorphic function, such that each $G_{n}$ and $R_{n}$ are
polynomials of degree at most $k-1$. This form is essentially unique.\end{namedthm}
\begin{rem}
\label{rem_about_sigma}~
\begin{enumerate}
\item All the results of the present paper remain valid for any choice of
$\sigma\in\ww N$ provided $\sigma+\mu\notin\ww R_{\leq0}$.
\item The normal forms presented above agree with the normal forms of \noun{J.~Écalle}
and of \noun{F.~Loray} when $k=1$, with $P:=1$ and $G:=0$.
\item The uniqueness clause of this theorem reflects the action of ${\tt Aut}_{k}$
on ${\tt Mod}_{k}$. We show in Proposition~\ref{pro:uniqueness}
and Corollary~\ref{cor:uniqueness} that two vector fields in normal
form are locally analytically conjugate if, and only if, the corresponding
triples $\left(P,G,R\right)$ are conjugate under the action of ${\tt Aut}_{k}$
by right-composition $\left(x,y\right)\mapsto\left(e^{\nf{2\ii\pi\theta}k}x,cy\right)$,
the element $\left(\theta,c\right)\in{\tt Aut}_{k}$ being the same
as the one defining the equivalence between the corresponding classification
moduli.
\end{enumerate}
\end{rem}
\bigskip{}

We mentioned earlier that our method is rather constructive. To underline
that fact we provide algorithms allowing us to prove computability
results, in the sense of the
\begin{defn}
\label{def_computability}~
\begin{enumerate}
\item We say that a number $x\in\ww R$ is \textbf{computable} if there
exists a halting Turing machine%
\footnote{We consider here Turing machines with finite alphabet and potentially
infinite memory.%
} ${\tt N}_{x}$ which inputs an integer $k$ and outputs a decimal
number $w\in10^{-k}\ww Z$ such that $\left|x-w\right|<10^{-k}$.
This definition is extended to points of $\ww R^{n}$ in the obvious
way.
\item We say that a function $f\,:\,\Omega\to\ww C^{m}$ defined on $\Omega\subset\ww C^{n}$
is \textbf{computable} if for each computable argument $\mathbf{x}\in\Omega$
the value $f\left(\mathbf{x}\right)$ is also computable in the following
sense: $f$ is uniquely determined by a halting Turing machine ${\tt F}_{f}$
which inputs ${\tt N}_{\mathbf{x}}$ and outputs ${\tt N}_{f\left(\mathbf{x}\right)}$.
\end{enumerate}
\end{defn}
\begin{rem}
\label{rem_integ_is_computable}Any path integral of a computable
function along a computable path is computable. More generally the
local integral curves of a computable vector field have a computable
parameterization.\end{rem}
\begin{namedthm}
[Computation Theorem] The modulus map $\mathcal{M}$ and the process
of reduction to normal form are explicitly%
\footnote{The word «explicitly» here means that we actually indicate a way to
do so.%
} computable, in the following sense (a formal class $\left(\mu,P\right)$
being fixed as well as the knowledge of $k$).
\begin{enumerate}
\item There exists an explicit halting Turing machine ${\tt Modulus}$ which
inputs the Turing machine ${\tt F}_{Z}$ of a computable vector field
$Z$ and outputs ${\tt F}_{\mathcal{M}\left(Z\right)}$.
\item There exists an explicit halting Turing machine ${\tt NormalForm}$
which inputs the Turing machine ${\tt F}_{M}$ of a computable modulus
$M\in{\tt Mod}_{k}$ and outputs ${\tt F}_{Z}$ where $Z$ is in the
form of the above Main Theorem.
\end{enumerate}
\end{namedthm}

\subsection{Description of the techniques and outline of the article}

~

The problem naturally splits into two very distinct tasks: find orbital
normal forms 
\begin{eqnarray*}
X_{R}\left(x,y\right) & = & x^{k+1}\pp x+y\left(1+\mu x^{k}+xR\left(x,x^{\sigma}y\right)\right)\pp y\,,
\end{eqnarray*}
\emph{i.e.} normal forms for the underlying foliation, then find temporal
normal forms $U_{G}X_{R}$, 
\begin{eqnarray*}
U_{G}\left(x,y\right) & = & \frac{P\left(x\right)}{1+xP\left(x\right)G\left(x,x^{\sigma}y\right)}\,,
\end{eqnarray*}
for a fixed foliation. The method we use here is different from the
abstract proofs given in the original papers~\cite{MaRa-SN} or~\cite{Loray}
for the orbital part, and in~\cite{MeshVoro,Tey-SN} for the temporal
part. In order to present the construction we need to say a few words
about how the moduli are related to the vector field. Before doing
so, however, we indicate how to reduce our results to the case $\sigma=0$
(that is, $\mu\notin\ww R_{\le0}$). This will notably lighten the
notations and improve the clarity of the presentation.

\subsubsection{Reduction of the proof}

Assume that the Main Theorem is valid for every $M\in{\tt Mod}_{k}$
with formal orbital modulus $\mu\notin\ww R_{\le0}$. Take $\tilde{\mu}\leq0$
and pick $\sigma\in\ww N$ such that $\mu:=\tilde{\mu}+\sigma$ is
positive. For a given $\tilde{M}\in{\tt Mod}_{k}$ with formal orbital
modulus $\tilde{\mu}$, define $M$ by replacing $\tilde{\mu}$ with
$\mu$. We transform the normal form $Z:=U_{G}X_{R}$ provided by
the Main Theorem with complete modulus $\mathcal{M}\left(Z\right)=M$
using the polynomial map 
\begin{eqnarray*}
\Psi\,:\,\ww C^{2} & \longrightarrow & \ww C^{2}\\
\left(x,y\right) & \longmapsto & \left(x,x^{\sigma}y\right)\,.
\end{eqnarray*}
This is a biholomorphism outside $\left\{ x=0\right\} $ such that
\begin{eqnarray*}
\Psi^{*}Z & = & U_{G\circ\Psi}\left(x^{k+1}\pp x+y\left(1+\tilde{\mu}x^{k}+xR\circ\Psi\right)\pp y\right)
\end{eqnarray*}
defines a germ of a holomorphic vector field $\tilde{Z}$ in normal
form with formal orbital modulus $\tilde{\mu}$. The fact that $\mathcal{M}\left(\tilde{Z}\right)=\tilde{M}$
will follow from the construction we present now, namely that the
non-formal moduli of $\mathcal{M}\left(\tilde{Z}\right)$ are completely
determined by the conformal structure of the dynamical system outside
$\left\{ x=0\right\} $. The uniqueness of the normal form follows
in the same way.

\subsubsection{The orbital modulus}

~

It is well-known that a convergent saddle-node is conjugate to some
prepared form, called Dulac's form~\cite{Dulac2} 
\begin{eqnarray}
Z\left(x,y\right) & = & u\left(x,y\right)\left(x^{k+1}\pp x+y\left(1+\mu x^{k}+x\, r\left(x,y\right)\right)\pp y\right)\label{eq:dulac}
\end{eqnarray}
with $u\left(0,0\right)\neq0$ and $r\left(x,0\right)=0$. This form
is far from being unique as $r$ and $u$ are otherwise unspecified
germs of a holomorphic function. The above vector field is orbitally
conjugate to the formal model $X_{0}$ over sector-like domains $\left(\mathcal{V}^{j}\right)_{j\in\zsk}$
by sectorial diffeomorphisms $\left(\Psi_{\tx O}^{j}\right)_{j\in\zsk}$,
see \cite{Tey-SN}. The union $\left\{ x=0\right\} \cup\bigcup_{j}\mathcal{V}^{j}$
is a neighborhood of $\left(0,0\right)$. For each $j$, there exists
a unique such conjugacy $Z=\Psi_{\tx O}^{j\ *}X_{0}$ which is tangent
to the identity and fibered in the $x$-variable. We will only use
these in the sequel:
\begin{eqnarray*}
\Psi_{\tx O}^{j}\,:\,\mathcal{V}^{j} & \longrightarrow & \ww C^{2}\\
\left(x,y\right) & \longmapsto & \left(x,y+\hot\right)\,.
\end{eqnarray*}

The formal model admits a family of sectorial first-integral with
connected fibers%
\footnote{A first-integral of $Z$ is a (perhaps multivalued) function $H$
such that $Z\cdot H=0$, which means the fibers $\left\{ H=\mbox{cst}\right\} $
are a union of integral curves of $Z$.%
}
\begin{eqnarray*}
H_{0}\left(x,y\right) & := & yx^{-\mu}\exp\frac{1}{kx^{k}}\,.
\end{eqnarray*}
We actually consider the following choices of sectorial determinations
of this (in general) multivalued function. Let $H_{0}^{0}$ denote
the determination of $H_{0}$ on\emph{ $\mathcal{V}^{0}$ }obtained
by taking the principal determination of the logarithm in $x^{-\mu}=\exp\left(-\mu\log x\right)$,
and for values of $j\in\ww Z$ set $H_{0}^{j+1}:=\exp\left(\nf{2\ii\pi\mu}k\right)H_{0}^{j}$
(computed from the analytic continuation of $H_{0}^{j}$ in $\mathcal{V}^{j+1}$).
Notice that then $H_{0}^{j}$ actually depends only on the class of
$j$ in $\zsk$. 

From this collection of sectorial functions we define a family of
(canonical) first-integral with connected fibers to the original vector
field $Z$ by letting
\begin{eqnarray*}
H^{j}:=H_{0}^{j}\circ\Psi_{\tx O}^{j} & \in & \mathcal{O}\left(\mathcal{V}^{j}\right)\,.
\end{eqnarray*}

\bigskip{}

The orbital analytic class of $Z$ is completely encoded in the way
sectorial leaves are connected over intersections of consecutive sectors
\begin{eqnarray*}
\mathcal{V}^{j,s} & := & \mathcal{V}^{j+1}\cap\mathcal{V}^{j}\,,
\end{eqnarray*}
namely we have the relation
\begin{eqnarray*}
H^{j+1} & = & \psi^{j}\circ H^{j}
\end{eqnarray*}
where $\left(\psi^{j}\right)_{j}$ is associated to the orbital modulus
$\left(\varphi^{j}\right)_{j\in\zsk}$ of $Z$ by 
\begin{eqnarray*}
\psi^{j}\left(h\right) & = & h\exp\left(\frac{2\ii\pi\mu}{k}+\varphi^{j}\left(h\right)\right)\,.
\end{eqnarray*}
Notice that the orbital modulus of $X_{0}$ is given by $\varphi^{j}=0$.

\subsubsection{Orbital normal forms: Section~\ref{sec:orbit-NF}}

~

Being given $\mu\notin\ww R_{\leq0}$ and $\left(\varphi^{j}\right)_{j\in\zsk}\in h\germ h$
we construct a collection $\left(H^{j}\right)_{j\in\zsk}$ of functions
with connected fibers such that $H^{j+1}=\psi^{j}\circ H^{j}$, holomorphic
on modified sectors $\mathcal{V}^{j}$ whose union covers $\ww C_{\neq0}\times\ww C$.
This is done by iterating a Cauchy-Heine integral transformation solving
a certain sectorial Cousin problem. The limit of the sequence obtained
in this way is a fixed-point of a certain operator between convenient
Banach spaces. We then associate to $H^{j}$ a sectorial vector field
\begin{eqnarray*}
X^{j}\left(x,y\right) & = & x^{k+1}\pp x+y\left(1+\mu x^{k}+xR^{j}\left(x,y\right)\right)\pp y
\end{eqnarray*}
such that $H^{j}$ is a first-integral of $X^{j}$. The construction
guarantees that $R^{j}$ is bounded near $\left\{ x=0\right\} $ and
$R^{j+1}=R^{j}$ on consecutive sectors. Riemann's Theorem on removable
singularities asserts that each $R^{j}$ is the restriction of a function
$R$ holomorphic on the domain
\begin{eqnarray*}
\mathcal{V}_{\rho} & := & \left\{ \left(x,y\right)\,:\,\left|y\right|<\rho\right\} 
\end{eqnarray*}
for some $\rho>0$. This means that the $X^{j}$ glue to a convergent
saddle-node vector field $X$. The growth of $R$ as $x\to\infty$
in $\mathcal{V}_{\rho}$ is also controlled by the Cauchy-Heine integral,
finally yielding that $X_{R}$ is in normal form.

\subsubsection{\label{sub:tempo_modulus}The temporal modulus}

~

From now on we assume that $Z=UX_{R}$ has a normalized orbital part
$X_{R}$. The formal normal form is then $P\, X_{R}$, where $P\left(x\right)$
is the polynomial of degree at most $k$ such that $P\left(x\right)\equiv U\left(x,0\right)\,\mbox{mod}\, x^{k+1}$.
There exist sectorial diffeomorphisms $\left(\Psi_{\tx T}^{j}\right)_{j}$
conjugating $Z$ to the vector field $PX_{R}$ (they are in particular
symmetries of the foliation induced by $X_{R}$). In a way, the temporal
modulus of $Z$ measures the obstruction to glue together the sectorial
flows of $PX_{R}$ in the intersections $\mathcal{V}^{j,s}$. This
invariant can be interpreted in terms of the period operator as the
obstruction to solve the cohomological equation 
\begin{eqnarray}
X_{R}\cdot T & = & \frac{1}{U}-\frac{1}{P}\,.\label{eq:cohomo-equa}
\end{eqnarray}
Let us explain the connection in some detail. By the method of characteristics,
any solution to the above cohomological equation must satisfy 
\begin{eqnarray*}
T\circ\gamma\left(1\right)-T\circ\gamma\left(0\right) & = & \tau\left(\frac{1}{U}-\frac{1}{P},\gamma\right)\,,
\end{eqnarray*}
 if $\gamma\,:\,\left[0,1\right]\to\left(\ww C^{2},0\right)$ is any
path tangent to $X_{R}$. Here
\begin{eqnarray*}
\tau\left(g,\gamma\right) & := & \int_{\gamma}g\left(x,y\right)\frac{\dd x}{x^{k+1}}\,.
\end{eqnarray*}
By using <<asymptotic paths>> $\gamma$, \emph{i.e.} satisfying
$\lim_{t\to0}\gamma\left(t\right)=\left(0,0\right)$, tangent to $X_{R}$,
we can solve the equation on $\mathcal{V}^{j}$ by a holomorphic function.
It follows that the cohomological equation admits a unique bounded
solution $T^{j}\in\mathcal{O}\left(\mathcal{V}^{j}\right)$ with continuous
extension to $\left\{ x=0\right\} $ such that $T^{j}\left(0,0\right)=0$.
This function provides the sectorial temporal normalization through
the relation
\[
\Psi_{\tx T}^{j}\left(x,y\right)=\flow{PX_{R}}{T^{j}\left(x,y\right)}{\left(x,y\right)}\,.
\]
This means that $\Psi_{\tx T}^{j}\left(x,y\right)$ is obtained by
replacing the time $t$ by $T^{j}\left(x,y\right)$ in the flow $\flow{PX_{R}}t{\left(x,y\right)}$
of the vector field $PX_{R}$. Since $\tau\left(\frac{1}{U},\gamma\right)$
represents the time needed to go from $\gamma\left(0\right)$ to $\gamma\left(1\right)$
following the flow of $Z$, the substitution $t=T^{j}(x,y)$ can be
naturally interpreted as a change of time in the flow of $PX_{R}$
in order to obtain that of $Z$. We refer to~\cite{Tey-SN} for details.

\bigskip{}

Now we identify the obstruction to solve (\ref{eq:cohomo-equa}) analytically
as the difference between consecutive sectorial solutions. Since this
difference is a first integral, it can be written 
\begin{eqnarray*}
T^{j+1}-T^{j} & = & f^{j}\circ H^{j}\,,
\end{eqnarray*}
where $H^{j}$ denotes the canonical first integral of $X_{R}$ on
$V^{j}$ introduced in subsection 1.2.2. The obstruction is thus located
in the value of the integral along an <<asymptotic cycle>> $\gamma_{p}^{j,s}$
passing through $p\in\mathcal{V}^{j,s}$ which is not homotopically
trivial in the leaf. We refer to Figure~\ref{fig:asy_cyc} for a
visual depiction.
\begin{defn}
\label{def_period}The function
\begin{eqnarray*}
p\in\mathcal{V}^{j,s} & \longmapsto & \tau\left(\frac{1}{U}-\frac{1}{P},\gamma_{p}^{j,s}\right)=f^{j}\circ H^{j}\left(p\right)\in\ww C
\end{eqnarray*}
 is called the \textbf{period} of $\frac{1}{U}-\frac{1}{P}$ along
$X_{R}$ on $\mathcal{V}^{j,s}$. Together, they define a linear operator
\begin{eqnarray*}
\per R\,:\,\frac{1}{U}-\frac{1}{P} & \longmapsto & \left(f^{j}\right)_{j\in\zsk}\in\left(h\germ h\right)^{k}\ .
\end{eqnarray*}
 
\end{defn}
As the asymptotic cycle is the same for all $p$ in $\mathcal{V}^{j,s}$
having the same value $h=H^{j}\left(p\right)$, we denote it by $\gamma^{j,s}\left(h\right)$.
The collection $\left(f^{j}\right)_{j}$ is the temporal modulus of
$Z$.
\begin{enumerate}
\item The sectorial solutions of $X_{R}\cdot T^{j}=g$ and the collection
$\per R\left(g\right)=(\per R^{j}\left(g\right))_{j\in\nf{\ww Z}{k\ww Z}}$
are defined in the same way for arbitrary germs in $g\in\ww C\left\{ x,y\right\} $
provided $g\left(x,0\right)\in x^{k+1}\ww C\left\{ x\right\} $. We
introduce the notation $x^{k+1}\germ x+y\germ{x,y}$ for the set of
these germs. For sufficiently small complex $h$, we have 
\[
\per R^{j}\left(g\right)\left(h\right)=\int_{\gamma^{j,s}\left(h\right)}g\left(x,y\right)\frac{\dd x}{x^{k+1}}\,
\]
 with the above asymptotic cycle in $\mathcal{V}^{j,s}$ corresponding
to the value $h$ of $H^{j}$.
\item The above considerations can be condensed in the statement that the
following sequence of vector spaces is exact
\begin{eqnarray*}
\begin{array}{ccccccccccc}
 &  &  & \mbox{const} &  & X_{R}\cdot &  & \per R\\
0 & \longrightarrow & \ww C & \longrightarrow & \germ{x,y} & \longrightarrow & x^{k+1}\germ x+y\germ{x,y} & \longrightarrow & \left(h\germ h\right)^{k} & \longrightarrow & 0
\end{array} &  & \,.
\end{eqnarray*}

\end{enumerate}
\begin{figure}
\hfill{}\includegraphics[height=5cm]{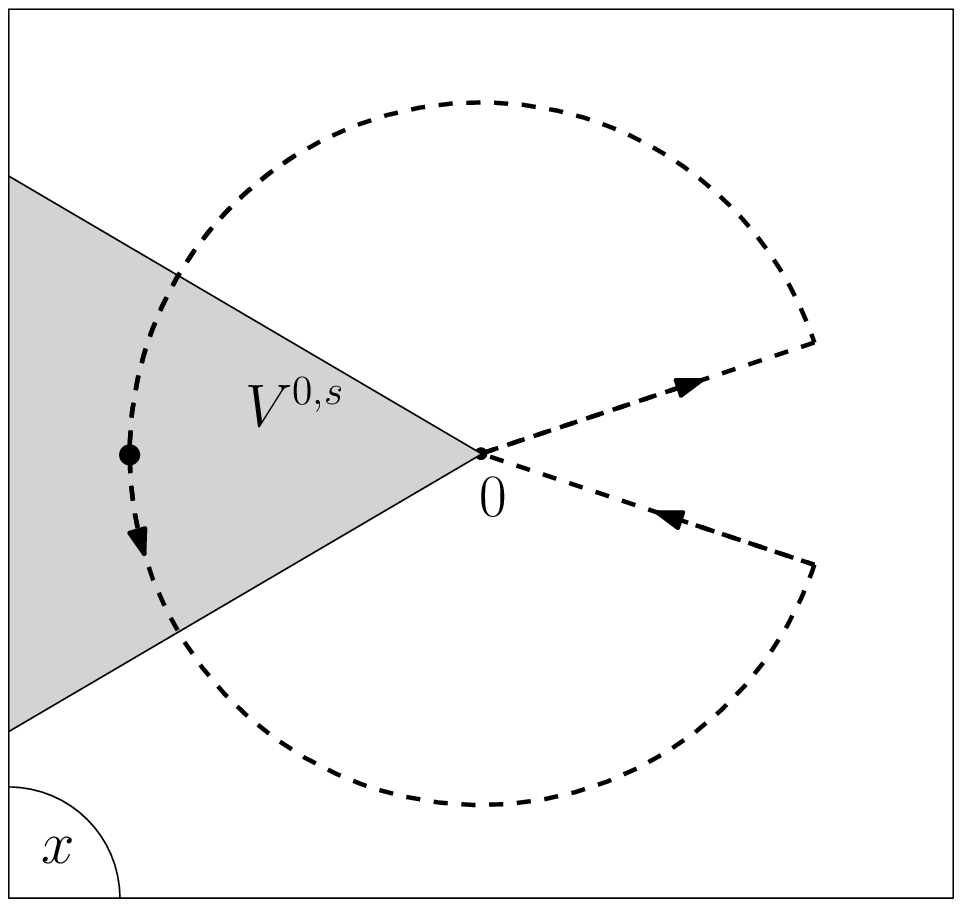}\hfill{}\includegraphics[height=5cm]{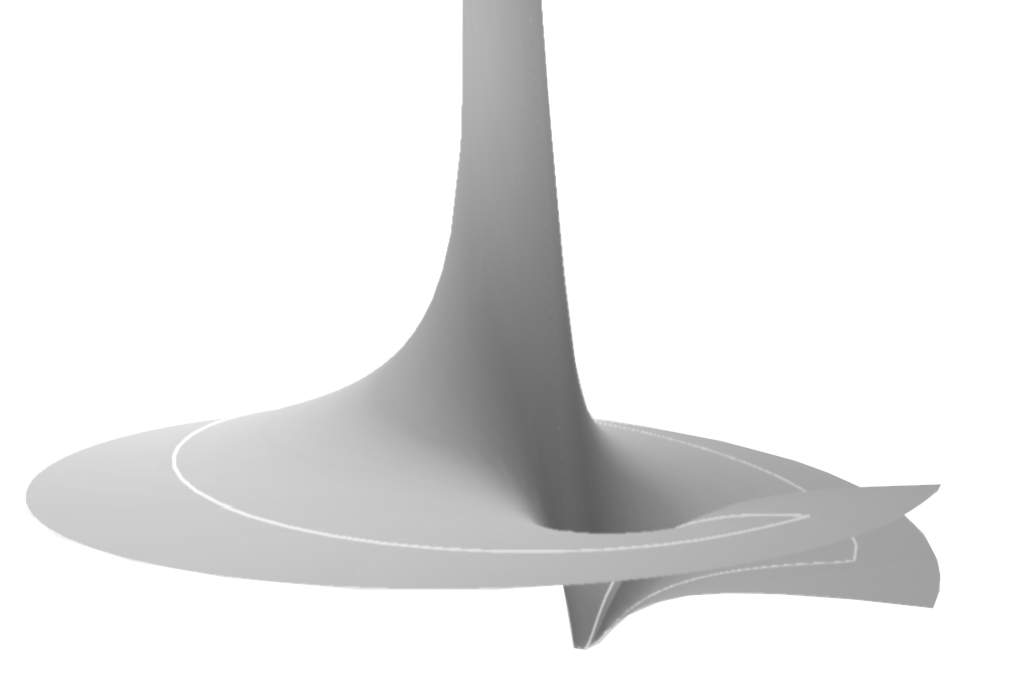}\hfill{}

\caption{\label{fig:asy_cyc}An asymptotic cycle, in projection in the $x$-variable
(left) and in the leaf (right) when $k=1$.}
\end{figure}

\subsubsection{Temporal normal forms: Section~\ref{sec:tempo-NF}}

~

Being given $\mu\notin\ww R_{\leq0}$ and $X_{R}$ from the previous
step of the construction, we consider some collection $\left(f^{j}\right)_{j\in\zsk}$.
Using again the Cauchy-Heine transformation, we obtain sectorial functions
$T^{j}\in\mathcal{O}\left(\mathcal{V}^{j}\right)$ such that $T^{j+1}-T^{j}=f^{j}\circ H^{j}$.
The construction ensures that $X_{R}\cdot T^{j+1}=X_{R}\cdot T^{j}$
and hence the functions $g^{j}=X_{R}\cdot T^{j}$ glue to a holomorphic
function $g$ for some $g$ with $g\in x\polg[y]$, where we define
$\polg[y]:=y\pol x{<k}\left\{ y\right\} $, the algebra of germs of
a holomorphic function of the form
\begin{eqnarray*}
\polg & = & \left\{ f\in\germ{x,y}\,:\, f\left(x,y\right)=\sum_{n>0}f_{n}\left(x\right)y^{n}\,\,,\, f_{n}\in\pol x{<k}\right\} \,.
\end{eqnarray*}
 Let $G:=\frac{g}{x}$; by construction $U_{G}X_{R}$ has the desired
temporal modulus. The construction yields as a by-product a \textbf{natural
section} $\sec R$ of the period operator
\begin{eqnarray*}
\sec R\,:\,\left(h\germ h\right)^{k} & \longrightarrow & x\polg\\
\left(f^{j}\right)_{j\in\zsk} & \longmapsto & g\,.
\end{eqnarray*}
The main difficulty here is to control the size of the domain on which
$\sec R\left(\left(f^{j}\right)_{j}\right)$ is holomorphic in terms
of that of the $f^{j}$.

\subsubsection{Explicit computations and algorithms: Section~\ref{sec:Explicit}}

~

Apart from numerical algorithms we establish in order to prove the
Computation~Theorem in Section~\ref{sub:Computing}, we also present
a way to perform symbolic calculations in Section~\ref{sub:symbolic_convergent}.
All these techniques are based on the fact that the orbital and temporal
modulus are expressed in terms of the period operator (Remark~\ref{rem_orbital_realization})
and its natural section. The period is nothing else than an integral
of an explicit differential form along a path tangent to the vector
field. Yet the key point allowing these computations to be carried
out is the fact that when $Z$ is a convergent vector field written
in Dulac's prepared form~(\ref{eq:Dulac}) then the correspondence
linking the Taylor coefficients of a function and that of its period
is block-triangular.

When $Z$ is divergent it is still possible to carry out explicit
numerical computations, as will be presented in our upcoming work,
although the symbolic side appears more difficult to fathom. This
difficulty is well-known to specialists, see for instance the discussion
appearing in~\cite{Ecal}.

\subsection{Notations and basic definitions}

~

Throughout the article we use the following notations and conventions:
\begin{itemize}
\item We use bold-typed letters to refer to vectors $\mathbf{z}=\left(z_{1},\ldots,z_{n}\right)\in\ww C^{n}$.
\item $\pol z{\leq k}$ is the algebra of polynomials in $z$ of degree
at most $k$. By extension $\pol z{<k}$ stands for $\pol z{\leq k-1}$
and so on.
\item $\frml{\mathbf{z}}$ is the algebra of formal power series in the
(multi)variable $\mathbf{z}$.
\item $\germ{\mathbf{z}}$ is the algebra of germs of a holomorphic function
near $\mathbf{0}$ in the (multi)variable $\mathbf{z}$.
\item If $\mathcal{U}$ is a domain of $\ww C^{n}$ let $\mathcal{O}\left(\mathcal{U}\right)$
denote the algebra of functions holomorphic on $\mathcal{U}$. Then
let $\mathcal{O}\left(\mathcal{U}\right)\left\{ y\right\} $ denote
the set of functions holomorphic on $\mathcal{U}\times r\ww D$ for
sufficiently small $r>0$ ; more precisely, it is the inductive limit
of the algebras $\mathcal{O}\left(\mathcal{U}\times r\ww D\right)$
as $r\to0$.
\item $Z$ is a saddle-node vector field near $\left(0,0\right)$ under
Dulac's prepared form~(\ref{eq:dulac}). The notation $X$ is in
general reserved to saddle-node vector fields whose $\pp x$-component
is a function of $x$ only.
\item $Z\cdot$ stands for the Lie derivative along $Z$, stably acting
on $\frml{x,y}$ and on $\germ{x,y}$.
\item $\left(k,\mu\right)\in\ww N_{>0}\times\ww C\backslash\ww R_{\leq0}$
is the formal orbital modulus of $Z$ while $P\in\pol x{\leq k}$
with $P\left(0\right)\neq0$ is its formal temporal modulus.
\item $\polg[y]:=y\pol x{<k}\left\{ y\right\} $ is the algebra of germs
of a holomorphic function of the form
\begin{eqnarray*}
\polg & = & \left\{ f\in\germ{x,y}\,:\, f\left(x,y\right)=\sum_{n>0}f_{n}\left(x\right)y^{n}\,\,,\, f_{n}\in\pol x{<k}\right\} \,.
\end{eqnarray*}

\item $\left(V^{j}\right)_{j\in\zsk}$ are the sectors in the $x$-variable
which covers $\ww C_{\neq0}$, see Definition~\ref{def_sectors}
and Figures~\ref{fig:local_sectors},~\ref{fig:sector_and_leaf}.
From these we construct
\begin{eqnarray*}
V^{j,s} & := & V^{j}\cap V^{j+1}\,.
\end{eqnarray*}

\item If $\mathcal{V}$ is a domain of $\ww C^{2}$ we define the associated
sectors $\mathcal{V}^{j}$ and $\mathcal{V}^{j,s}$ as $\mathcal{V}\cap\left(V^{j}\times\ww C\right)$
and $\mathcal{V}\cap\left(V^{j,s}\times\ww C\right)$ respectively,
for $j\in\zsk$. One kind of domain will be of special interest:
\begin{eqnarray*}
\mathcal{V}_{\rho} & := & \ww C\times\rho\ww D=\left\{ \left(x,y\right)\,:\,\left|y\right|<\rho\right\} 
\end{eqnarray*}
 for $\rho>0$.
\item $X_{R}$ is the vector field associated to some $R\in y\ww C\left\{ x,y\right\} $
by
\begin{eqnarray*}
X_{R} & := & x^{k+1}\pp x+y\left(1+\mu x^{k}+xR\right)\pp y.
\end{eqnarray*}
It represents one of the normal forms given in Theorem \ref{thm:convergent},
if $R\in\polg[x^{\sigma}y]$. Observe that for $R=0$, we obtain the
formal model $X_{0}=x^{k+1}\pp x+y\left(1+\mu x^{k}\right)\pp y$.
\item $N=\left(N^{j}\right)_{j\in\zsk}$ is the collection of sectorial
normalizing functions for $X_{R}$, that is functions $N^{j}\in y\mathcal{O}\left(V^{j}\right)\left\{ y\right\} $
such that $\left(x,y\exp N^{j}\right)^{*}X_{0}=X_{R}$.
\item $H_{N}^{j}\in y\mathcal{O}\left(V^{j}\right)\left\{ y\right\} ,\ j\in\zsk$,
are the sectorial first-integral associated to $X_{R}$ (Definition~\ref{def_first-integral}):
\begin{eqnarray*}
H_{N}^{j}\left(x,y\right) & := & y\, e^{2i\pi\mu j/k}\exp\left(\frac{x^{-k}}{k}-\mu\log x+N^{j}\left(x,y\right)\right)\,.
\end{eqnarray*}
 Here the branch of the logarithm is chosen according to the sector,
\emph{i.e.} such that $\left|\arg x-j\frac{2\pi}{k}\right|\leq\frac{\pi}{k}+\beta$
for small $x$. 
\item $\per R$ is the period operator associated to $X_{R}$ (see Definition~\ref{def_period})
and $\sec R$ its natural section (see Corollary~\ref{prop:natural_section})
\begin{eqnarray*}
\per R\,:\, x^{k+1}\germ x+y\germ{x,y} & \longrightarrow & \left(h\germ h\right)^{k}\\
\sec R\,:\,\left(h\germ h\right)^{k} & \longrightarrow & x\polg\,.
\end{eqnarray*}

\item $\mathcal{M}\left(Z\right)=\left(\mu,P\right)\oplus\left(\varphi^{j},f^{j}\right)_{j\in\zsk}$
is the complete analytic modulus of $Z$. If $Z=U_{G}X_{R}$ is the
associated normal form then
\begin{eqnarray*}
\left(\varphi^{j},f^{j}\right)_{j\in\zsk} & = & \per R\left(-xR\right)\oplus\per R\left(\frac{1}{U_{G}}-\frac{1}{P}\right)\,.
\end{eqnarray*}

\end{itemize}
We introduce also some Banach spaces and norms.
\begin{defn}
\label{def_Banach}Let $\mathcal{D}\subset\ww C^{n}$ be a domain
containing $\mathbf{0}$ equipped with the coordinate $\mathbf{z}=\left(z_{1},\cdots,z_{n}\right)$. 
\begin{enumerate}
\item We define the Banach space $\mathcal{B}\left(\mathcal{D}\right)$
of functions bounded and holomorphic on $\mathcal{D}$ with values
in $\ww C$ equipped with the norm:
\begin{eqnarray*}
\norm f{\mathcal{D}} & := & \sup_{\mathbf{z}\in\mathcal{D}}\left|f\left(\mathbf{z}\right)\right|\,.
\end{eqnarray*}

\item We define the Banach space $\mathcal{B}'\left(\mathcal{D}\right)$
of functions holomorphic on $\mathcal{D}$, vanishing along $\left\{ z_{n}=0\right\} $,
equipped with the norm
\begin{eqnarray*}
\norm f{\mathcal{D}}' & := & \sup_{\mathbf{z}\in\mathcal{D}}\frac{\left|f\left(\mathbf{z}\right)\right|}{\left|z_{n}\right|}\,.
\end{eqnarray*}
Notice that
\begin{eqnarray*}
\norm f{\mathcal{D}}' & \leq & \norm{\frac{\partial f}{\partial z_{n}}}{\mathcal{D}}
\end{eqnarray*}
when $\frac{\partial f}{\partial z_{n}}\in\mathcal{B}\left(\mathcal{D}\right)$. 
\item For a finite collection $\mathcal{D}:=\left(\mathcal{D}_{j}\right)_{j}$,let
$\mathcal{B}\left(\mathcal{D}\right)$denote the Banach space $\prod_{j}\mathcal{B}\left(\mathcal{D}_{j}\right)$
equipped with the norm
\begin{eqnarray*}
\norm{\left(f_{j}\right)_{j}}{\mathcal{D}} & := & \max_{j}\norm{f_{j}}{\mathcal{D}_{j}}\,.
\end{eqnarray*}
The analogous definition is used for the space $\mathcal{B}'\left(\mathcal{D}\right)$.
\end{enumerate}
\end{defn}
In general we omit to indicate the dependence of the norm on the domain
when the context is not ambiguous.

\section{\label{sec:orbit-NF}Orbital normal forms}

Recall that in the following section $\mu$ is a non-zero, non-negative
complex number.

\subsection{Sectorial decomposition and first-integrals}

~

\begin{figure}
\hfill{}\includegraphics[width=12cm]{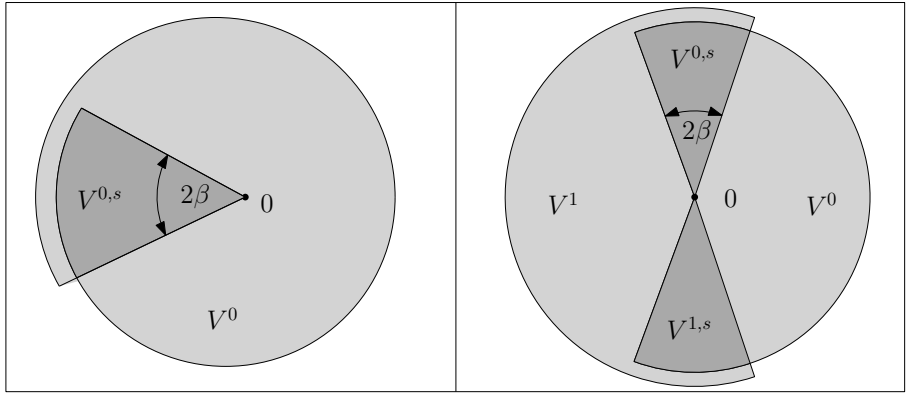}\hfill{}

\caption{\label{fig:local_sectors}Sectors near $0$ in the case $k=1$ (left)
along with the case $k=2$ (right)}
\end{figure}
\begin{figure}
\hfill{}\includegraphics[width=10cm]{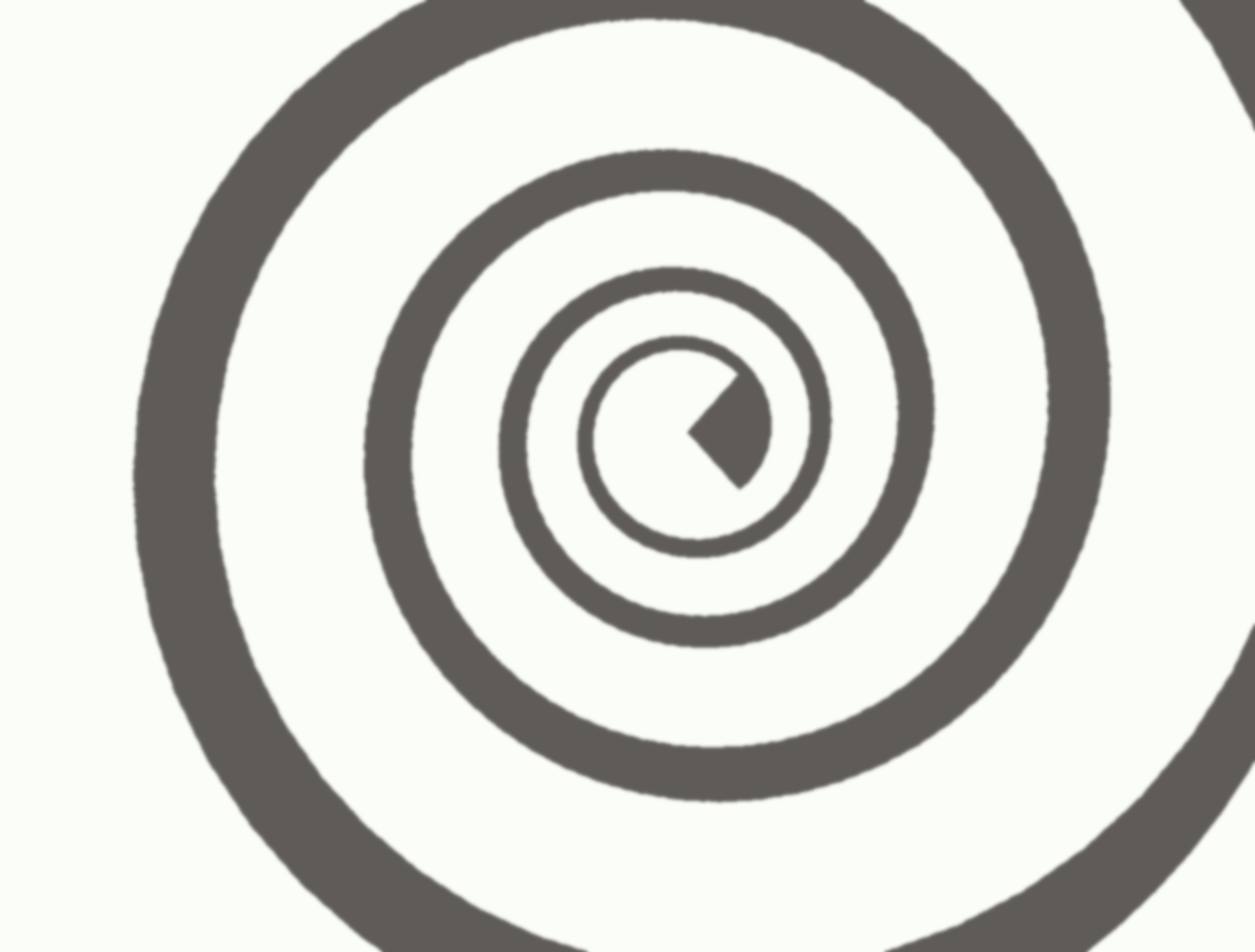}\hfill{}

\hfill{}\includegraphics[width=10cm]{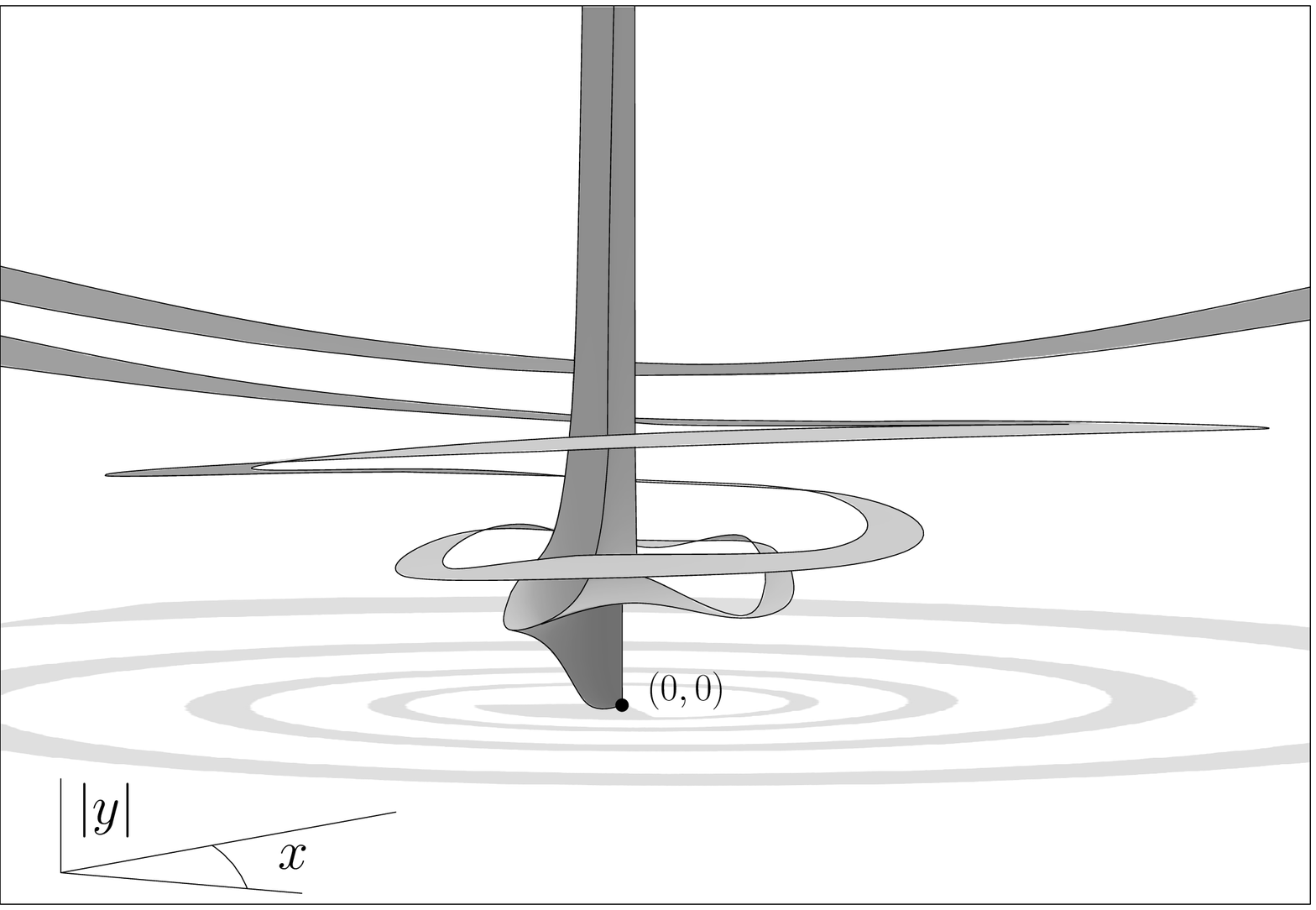}\hfill{}

\caption{\label{fig:sector_and_leaf}The sector $V^{0}$ in the $x$-variable
(top) and the absolute value of a sectorial leaf of the formal model
over it (bottom) for $k=3$ and $\mu=-\frac{1}{10}\left(1+\ii\right)$.}
\end{figure}

We fix once and for all a real number
\begin{eqnarray*}
0 & <\beta< & \frac{\pi}{2k}.
\end{eqnarray*}

\begin{defn}
\label{def_sectors}This definition should be read with the Figures~\ref{fig:local_sectors}~and~\ref{fig:sector_and_leaf}
in mind.
\begin{enumerate}
\item The sectorial decomposition of the $x$-variable is the collection
of $k$ sector-like domains $V^{j}$ defined as the union of 

\begin{itemize}
\item a standard sector of $r\ww D$
\begin{eqnarray*}
 & \left\{ x\,:\,\left|\arg x-\theta^{j}\right|<\frac{\pi}{k}+\beta\,\,\,\textrm{and}\,\,\,0<\left|x\right|<r\right\}  & \,\,\,,\, j\in\zsk
\end{eqnarray*}
where $\theta^{j}:=j\frac{2\pi}{k}$
\item a spiraling sector, bounded by two spirals 
\begin{eqnarray*}
S^{j,\pm} & :=r\exp\left(\ii\theta^{j}\pm\ii\frac{\pi}{k}\pm\ii\beta\right)\exp\left(\left(1+\ii\nu\right)\ww R_{\geq0}\right) & \,\,\,,\, j\in\zsk
\end{eqnarray*}
where $\nu\in\ww R$ is chosen (once and for all) in such a way that
\begin{eqnarray}
\re{\mu} & > & \nu\im{\mu}\,.\label{eq:asymptotic_spiral_behavior}
\end{eqnarray}
 In particular when $\re{\mu}>0$ we can take $\nu:=0$ and the sector
$V^{j}$ coincides with a standard sector of infinite radius.
\end{itemize}

Notice that $\bigcup_{j\in\zsk}V^{j}=\ww C\backslash\left\{ 0\right\} $.

\item We denote by $\Gamma^{j,\pm}$ the two connected components of $\partial V^{j}$,
consisting of the concatenation of the segment $I^{j,\pm}:=]0,r\exp\left(\ii\theta^{j}\pm\ii\frac{\pi}{k}\pm\ii\beta\right)]$
and the spiral $S^{j,\pm}$.
\item We define the\textbf{ saddle-part} of $V^{j}$ as
\begin{eqnarray*}
V^{j,s} & := & V^{j}\cap V^{j+1}\,.
\end{eqnarray*}

\item For any domain $\mathcal{V}\subset\ww C^{2}$ containing $\left\{ y=0\right\} $
we define the \textbf{sectorial decomposition of $\mathcal{V}$} by
\begin{eqnarray*}
\mathcal{V}^{j} & := & \mathcal{V}\cap\left(V^{j}\times\ww C\right)\,\,\,\,\,,\, j\in\zsk\\
\mathcal{V}^{j,s} & := & \mathcal{V}\cap\left(V^{j,s}\times\ww C\right)\,\,\,\,\,,\, j\in\zsk
\end{eqnarray*}

\end{enumerate}
\end{defn}
\begin{rem}
\label{rem_k=00003D1}In the case $k=1$ there is a slight problem
in the definition of $V^{0}$. We make the convention that $V^{0}$,
near $0$, is a sector of aperture greater than $2\pi$ which overlaps
with itself above $\ww R_{<0}$ without gluing (see Figure~\ref{fig:local_sectors}).
Let $V^{0,s}$ denote this overlap in the case $k=1$.\end{rem}
\begin{defn}
\label{def_first-integral}Let $\mathcal{V}\subset\ww C^{2}$ be a
domain containing $\left\{ y=0\right\} $.

For every collection $N:=\left(N^{j}\right)_{j\in\zsk}$in $\mathcal{B}\left(\left(\mathcal{V}^{j}\right)_{j}\right)$
we define the collection $\left(H_{N}^{j}\right)_{j\in\zsk}$ of $k$
functions by 
\begin{eqnarray*}
H_{N}^{j}\left(x,y\right) & := & y\, e^{2\ii\pi j\mu/k}\exp\left(\frac{1}{kx^{k}}-\mu\log x+N^{j}\left(x,y\right)\right)\,,
\end{eqnarray*}
where we choose the determination of the logarithm such that $\left|\arg x-j\frac{2\pi}{k}\right|\leq\frac{\pi}{k}+\beta$
for small $x\in V^{j}$ and such that it is an analytic function in
this «spiraling sector». In this way, the value of $H_{N}^{j}(x,y)$
indeed only depends upon the class of $j$ modulo $k$. 

This collection is called the \textbf{sectorial first-integrals}\emph{
}associated to $N$.
\end{defn}

\subsection{Admissible domains and the refined Cauchy-Heine transform}

~
\begin{defn}
\label{def_adapted_sectors}Let $\rho\in]0,+\infty]$ be given.
\begin{enumerate}
\item We define the domain 
\begin{eqnarray*}
\mathcal{V}_{\rho} & := & \ww C\times\rho\ww D=\left\{ \left(x,y\right)\in\ww C^{2}\,:\,\left|y\right|<\rho\right\} \,,
\end{eqnarray*}
which is a neighborhood of $\left\{ y=0\right\} $. 
\item A collection $\Delta=\left(\Delta^{j}\right)_{j\in\zsk}$ of $k$
domains of $\ww C$ containing $0$ will be called \emph{admissible}.
\item We say that a couple $\left(\rho,N\right)$ with $N=\left(N^{j}\right)\in\mathcal{B}\left(\left(\mathcal{V}_{\rho}^{j}\right)_{j\in\zsk}\right)$
is \emph{adapted} to an admissible collection $\Delta$ if $H_{N}^{j}\left(\mathcal{V}_{\rho}^{j,s}\right)\subset\Delta^{j}$
for each $j\in\zsk$.
\end{enumerate}
\end{defn}
The next result is the basis of our construction.
\begin{thm}
\label{thm:Cauchy-Heine}Consider some admissible collection $\Delta$
and some $\left(\rho,N\right)$ adapted to $\Delta$ with $\rho<+\infty$.
Take any collection $f=\left(f^{j}\right)_{j\in\zsk}\in\mathcal{B}'\left(\Delta\right)$
and define the collection $\Sigma=\left(\Sigma^{j}\right)_{j\in\zsk}$
by
\begin{eqnarray*}
\Sigma^{j}\left(x,y\right) & := & \frac{x}{2\ii\pi}\sum_{\ell\neq j+1}\int_{\Gamma^{\ell-}}\frac{f^{\ell}\left(H_{N}^{\ell}\left(z,y\right)\right)}{z\left(z-x\right)}\dd z+\frac{x}{2\ii\pi}\int_{\Gamma^{j+}}\frac{f^{j}\left(H_{N}^{j}\left(z,y\right)\right)}{z\left(z-x\right)}\dd z,\ \ (x,y)\in\mathcal{V}_{\rho}^{j},
\end{eqnarray*}
where the paths were described in Definition \ref{def_sectors}. Here
integrals over $\Gamma^{\ell\pm}$ are more precisely integrals in
$V^{\ell}$over paths arbitrarily close to $\Gamma^{\ell\pm}$. The
following properties hold.
\begin{enumerate}
\item $\Sigma^{j}\in\mathcal{B}\left(\mathcal{V}^{j}\right)$.
\item For all $j\in\zsk$ we have 
\begin{eqnarray*}
\Sigma^{j+1}-\Sigma^{j} & = & f^{j}\circ H_{N}^{j}
\end{eqnarray*}
on $\mathcal{V}_{\rho}^{j,s}$.
\item For every $j\in\zsk$
\begin{eqnarray*}
\lim_{x\to0}\Sigma^{j}\left(x,y\right) & = & 0
\end{eqnarray*}
locally uniformly in $y\in$$\rho\ww D$.
\item Any other collection $\left(\tilde{\Sigma}^{j}\right)_{j\in\zsk}\in\mathcal{B}\left(\left(\mathcal{V}_{\rho}^{j}\right)_{j\in\zsk}\right)$
satisfying~(1) and~(2) differs from $\Sigma$ by the component-wise
addition of a single holomorphic function $y\mapsto F\left(y\right)$
in $\mathcal{B}\left(\rho\ww D\right)$.
\item One has the estimates

\begin{enumerate}
\item 
\begin{eqnarray*}
\norm{\Sigma}{} & \leq & \rho K\norm f{}'e^{\norm N{}}
\end{eqnarray*}

\item 
\begin{eqnarray*}
\norm{y\frac{\partial\Sigma}{\partial y}}{} & \leq & \rho K\norm{f'}{}e^{\norm N{}}\left(1+\norm{y\frac{\partial N}{\partial y}}{}\right)
\end{eqnarray*}

\item 
\begin{eqnarray*}
\norm{x\frac{\partial\Sigma}{\partial x}}{} & \leq & \rho K\norm{f'}{}e^{\norm N{}}\left(1+\norm{x\frac{\partial N}{\partial x}}{}\right)
\end{eqnarray*}
with some constant $K>0$ depending only on $k$, $\mu$, $\nu$,
$\beta$ and $r$.
\end{enumerate}
\end{enumerate}
\end{thm}
\begin{rem}
A value for $K$ is given in the proof, but not very explicitly.\end{rem}
\begin{defn}
\label{def_RCH}Under the hypothesis of the theorem, we let the \textbf{refined
Cauchy-Heine transform} of $f$, associated to $N$, denote the collection
of functions $\Sigma\left(N,f\right):=\left(\Sigma^{j}\right)_{j\in\zsk}$
defined by the previous theorem. We choose this set of functions satisfying~(2)
because they are <<normalized>>: they tend to $0$ as $V^{j}\ni x\to0$.
\end{defn}
Let us now give a proof to Theorem~\ref{thm:Cauchy-Heine}.
\begin{proof}
In order to prove the holomorphy of $\Sigma^{j}$ and as a first step
to establish the estimates (5), we only consider the case of the integral
along $\Gamma^{j,+}$. For the sake of clarity we omit, wherever not
confusing, to indicate the indexes $j$, $N$ and $+$. 

By the above Definitions~\ref{def_Banach} and~\ref{def_first-integral},
one has 
\begin{eqnarray}
\left|x\frac{f\left(H\left(z,y\right)\right)}{z(z-x)}\dd z\right| & \leq & \norm f{}'\left|xH\left(z,y\right)\frac{\dd z}{z(z-x)}\right|\label{eq:maj_integ}\\
 & \leq & \rho A\norm f{}'e^{\norm N{}}q\left(x\right)\left|z^{-\mu-1}\exp\frac{z^{-k}}{k}\right|\left|\dd z\right|
\end{eqnarray}
where $q\left(x\right)=\sup\left\{ \frac{|x|}{|z-x|}\text{ }:\text{ }z\in\Gamma\right\} $
and $A=e^{2\pi\left|\mu\right|}$. The first thing to establish is
that the improper integrals in the construction of $\Sigma$ are absolutely
convergent. Using the notation introduced in Definition~\ref{def_sectors},
we have for $z\in I$ 
\begin{eqnarray*}
z & = & x_{*}t\,\,\,,\, t\in\left[0,1\right]\\
\left|\dd z\right| & = & r\dd t
\end{eqnarray*}
where 
\begin{eqnarray*}
x_{*} & := & r\exp\left(\ii\frac{\pi}{k}\left(2j+1\right)+\ii\beta\right)\,,
\end{eqnarray*}
while for $z\in S$
\begin{eqnarray*}
z & = & x_{*}\exp\left(\left(1+\ii\nu\right)t\right)\,\,\,,\, t\geq0\\
\left|\dd z\right| & = & \left|1+\ii\nu\right|\left|z\right|\dd t\,.
\end{eqnarray*}
For $z\in I$ we have $\arg\left(z^{k}\right)=\pi+k\beta$ and therefore
\begin{eqnarray*}
\left|z^{-\mu-1}\exp\left(\frac{z^{-k}}{k}\right)\dd z\right| & \leq & \left|x_{*}^{-\mu-1}\right|t^{-\re{\mu}-1}\exp\left(-r^{-k}\cos\left(k\beta\right)t^{-k}/k\right)\dd t
\end{eqnarray*}
is integrable on $\left[0,1\right]$. For $z$ is on $S$ we have
$\left|z^{-\mu-1}\right|=\left|x_{*}^{-\mu-1}\right|e^{\alpha t}$
with 
\begin{eqnarray*}
\alpha & :=-\re{\mbox{\ensuremath{\left(\mu+1\right)\left(1+i\nu\right)}}}= & -\re{\mu}-1+\nu\im{\mu}<-1
\end{eqnarray*}
(see Definition \ref{def_sectors}) and therefore
\begin{eqnarray*}
\left|z^{-\mu-1}\exp\left(\frac{z^{-k}}{k}\right)\dd z\right| & \leq & \left|1+\nu\ii\right|\exp\frac{r^{-k}}{k}\left|x_{*}^{-\mu}\right|e^{\left(\alpha+1\right)t}\dd t
\end{eqnarray*}
is also integrable on $[0,+\infty[$. Hence 
\[
\left|\frac{x}{2\ii\pi}\int_{\Gamma}\frac{f\left(H\left(z,y\right)\right)}{z\left(z-x\right)}\dd z\right|\leq\rho AL\norm f{}'e^{\norm N{}}q\left(x\right),
\]
where $L=\frac{1}{2\pi}\int_{\Gamma}\left|z^{-\mu-1}\exp\frac{z^{-k}}{k}\right|\left|\dd z\right|.$
In the case of $\mu$ with negative real part, it is crucial to use
the spiral shape of the paths near $\infty$ as required by~(\ref{eq:asymptotic_spiral_behavior})
of Definition~\ref{def_sectors}. 

\begin{figure}
\hfill{}\includegraphics[width=12cm]{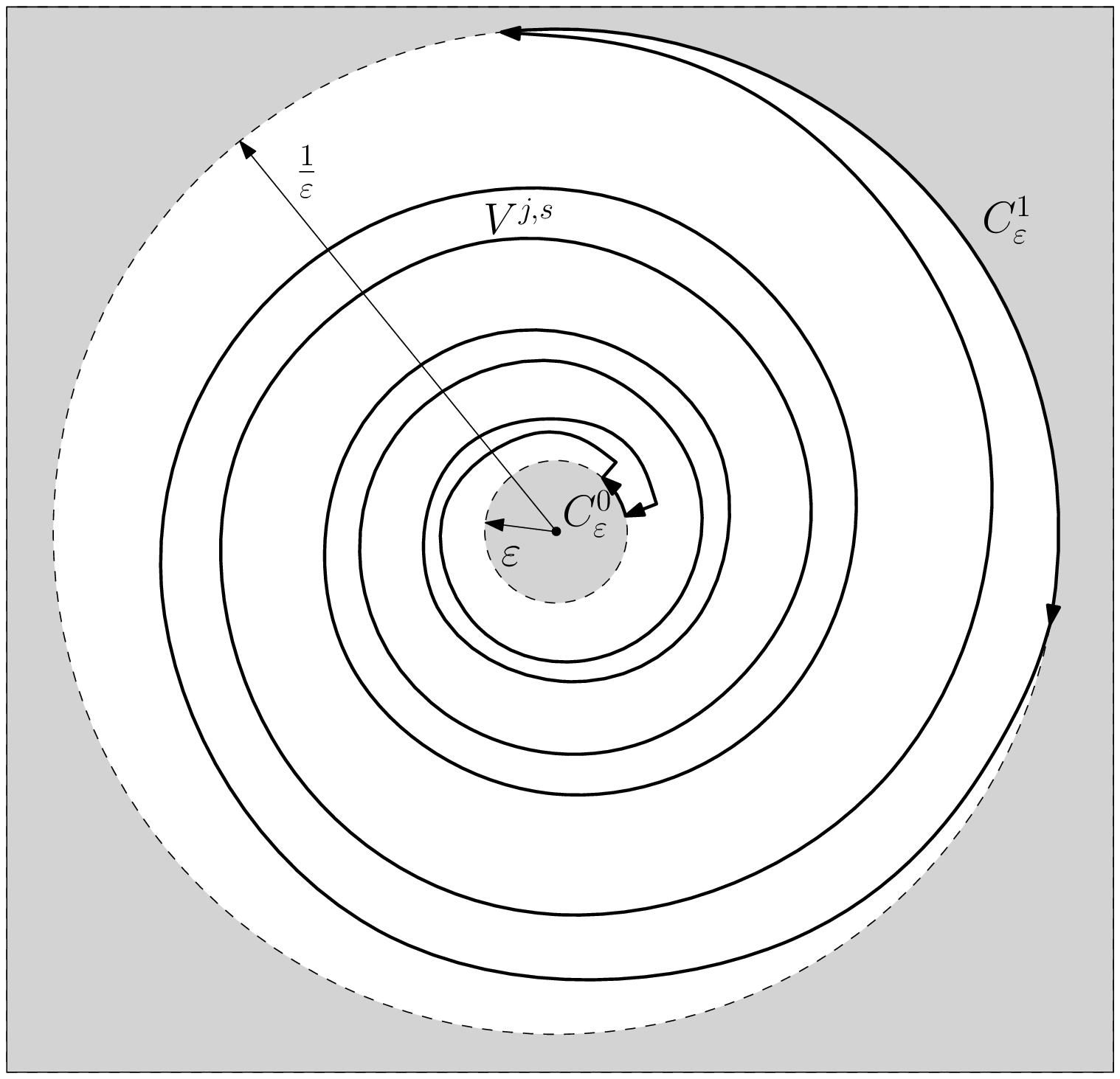}\hfill{}

\caption{\label{fig:Cauchy_contour}}
\end{figure}

By choosing the paths of integration sufficiently close to the boundaries
of the sectors, we obtain that $\Sigma^{j}$ is an analytic function
on $\mathcal{V}_{\rho}^{j}$. The boundedness of $\Sigma$ will be
shown later; only then the proof of (1) is complete.

Claim~(2) is obtained by Cauchy's formula. For $\varepsilon>0$ small
enough, we define the contour $\mathcal{C}_{\varepsilon}$ as in Figure~\ref{fig:Cauchy_contour}.
It is positively oriented and consists of:
\begin{itemize}
\item the arc $C_{\varepsilon}^{0}$ of the circle $\left\{ \left|z\right|=\varepsilon\right\} $
between $\Gamma^{j,+}$ and $\Gamma^{j+1,-}$, 
\item the curve $\Gamma^{j+1,-}\cap\left\{ \varepsilon\leq\left|z\right|\leq\frac{1}{\varepsilon}\right\} $,
\item the arc $C_{\varepsilon}^{1}$ of the circle $\left\{ \left|z\right|=\frac{1}{\varepsilon}\right\} $
between $\Gamma^{j,+}$ and $\Gamma^{j+1,-}$,
\item the curve $\Gamma^{j,+}\cap\left\{ \varepsilon\leq\left|z\right|\leq\frac{1}{\varepsilon}\right\} $,
.
\end{itemize}
Therefore whenever $x\in V^{j,s}$, $\varepsilon<\left|x\right|<\frac{1}{\varepsilon},$
we have 
\begin{eqnarray*}
\frac{1}{2\ii\pi}\int_{\mathcal{C}_{\varepsilon}}\frac{f^{j}\left(H_{N}^{j}\left(z,y\right)\right)}{z\left(z-x\right)}\dd z & = & \frac{f^{j}\left(H_{N}^{j}\left(x,y\right)\right)}{x}\,.
\end{eqnarray*}
Taking the limit as $\varepsilon\to0$ yields that 

\begin{equation}
\frac{x}{2\ii\pi}\int_{\Gamma^{j+1,-}}\frac{f^{j}\left(H_{N}^{j}\left(z,y\right)\right)}{z\left(z-x\right)}\dd z-\frac{x}{2\ii\pi}\int_{\Gamma^{j+}}\frac{f^{j}\left(H_{N}^{j}\left(z,y\right)\right)}{z\left(z-x\right)}\dd{z=f^{j}\left(H_{N}^{j}\left(x,y\right)\right)}\label{eq:cauchy-gamma}
\end{equation}
because the integrals on $C_{\varepsilon}^{0}$ and $C_{\varepsilon}^{1}$
tend to $0$. Indeed, for values of $\varepsilon$ less than $\min\left(r,\frac{\left|x\right|}{2}\right)$,
if $z\in C_{\varepsilon}^{0}$, then we have $\left|\arg z-\left(2j+1\right)\frac{\pi}{k}\right|<\beta$
and $\left|z\left(z-x\right)\right|\geq\varepsilon\left|\frac{x}{2}\right|$
. Therefore
\begin{eqnarray*}
\int_{C_{\varepsilon}^{0}}\left|\frac{xz^{-\mu-1}}{z-x}\exp\left(\frac{z^{-k}}{k}\right)\dd z\right| & \leq & 4\beta\varepsilon^{-\re{\mu}}e^{2\pi\left|\mu\right|}\exp\left(-\frac{\varepsilon^{-k}}{k}\cos\left(k\beta\right)\right)\,.
\end{eqnarray*}
If $z\in C_{\varepsilon}^{1}$ we still have $\left|z-x\right|\geq\left|x\right|$
and therefore
\begin{eqnarray*}
\int_{C_{\varepsilon}^{1}}\left|\frac{xz^{-\mu-1}}{z-x}\exp\left(\frac{z^{-k}}{k}\right)\dd z\right| & \leq & 2\beta\, r^{-\nu\im{\mu}}e^{2\pi\left|\mu\right|}\varepsilon^{-\alpha-1}\exp\left(\frac{\varepsilon^{k}}{k}\right)\,.
\end{eqnarray*}
In a similar way, we show that 
\[
\frac{x}{2\ii\pi}\int_{\Gamma^{j+1,+}}\frac{f^{j+1}\left(H_{N}^{j+1}\left(z,y\right)\right)}{z\left(z-x\right)}\dd z=\frac{x}{2\ii\pi}\int_{\Gamma^{j+2,-}}\frac{f^{j+1}\left(H_{N}^{j+1}\left(z,y\right)\right)}{z\left(z-x\right)}
\]
for $\left(x,y\right)\in\mathcal{V}^{j+1,s}$. This completes the
proof of~(2). 

Next we show (5)(a) (and at the same time the boundedness part of
(1)). For that purpose, we first consider the subset $W^{j}$ of $V^{j}$
containing all $x$ such that $xe^{\ii s}\in V^{j}$ for real $s,$
$0\leq s\leq\beta/2.$ Then $q_{j}\left(x\right)\leq1/\sin\left(\beta/2\right)$
if $\left|x\right|$ is small whereas $1/q_{j}\left(x\right)$ is
bounded below by the distance of the point $e^{\ii\beta/2}$ to the
spiral $\exp\left(\left(1+\ii\nu\right)\ww R\right)$ for large $\left|x\right|.$
As $q_{j}$ is a continuous function, this implies that it is bounded
on $W^{j}$$.$ Let $Q$ denote some bound. Thus we have shown that
\[
\left|\frac{x}{2\ii\pi}\int_{\Gamma^{j,+}}\frac{f^{j}\left(H_{N}^{j}\left(z,y\right)\right)}{z\left(z-x\right)}\dd z\right|\leq\rho ALQ\norm f{}'e^{\norm N{}}
\]
for $x\in W^{j}$ and $\left|y\right|\leq\rho$. For $x\in V^{j}\setminus W^{j}$,
we use (\ref{eq:cauchy-gamma}) and estimate the integral over $\Gamma^{j+1,-}$
in the same way. This yields 
\[
\left|\frac{x}{2\ii\pi}\int_{\Gamma^{j,+}}\frac{f^{j}\left(H_{N}^{j}\left(z,y\right)\right)}{z\left(z-x\right)}\dd z\right|\leq\rho A(LQ+M)\norm f{}'e^{\norm N{}}
\]
 for $x\in V^{j},\left|y\right|\leq\rho$, where $M$ denotes the
supremum of $\left|x^{-\mu-1}\exp\frac{x^{-k}}{k}\right|$ on $V^{j}.$
The remaining integrals in the definition of $\Sigma^{j}$ are treated
similarly. This yields (5)(a).

From the above estimates, we can also deduce~(5)(b): since

\[
y\frac{\partial H}{\partial y}\left(x,y\right)=H\left(x,y\right)\left(1+y\frac{\partial N}{\partial y}\left(x,y\right)\right)
\]
we use

\begin{eqnarray*}
\left|x\frac{f'\left(H\left(z,y\right)\right)}{z(z-x)}y\frac{\partial H}{\partial y}\left(z,y\right)dz\right| & \leq & \rho A\norm{f'}{}e^{\norm N{}}\norm{1+y\frac{\partial N}{\partial y}}{}q\left(x\right)\left|z^{-\mu-1}\exp\frac{z^{-k}}{k}\right|\left|\dd z\right|
\end{eqnarray*}
instead of (\ref{eq:maj_integ}). 

For (5)(c), observe that 
\[
x\frac{\partial\Sigma^{j}}{\partial x}\left(x,y\right)=\Sigma^{j}(x,y)+\frac{x^{2}}{2\ii\pi}\sum_{\ell\neq j+1}\int_{\Gamma^{\ell-}}\frac{f^{\ell}\left(H_{N}^{\ell}\left(z,y\right)\right)}{z\left(z-x\right)^{2}}\dd z+\frac{x^{2}}{2\ii\pi}\int_{\Gamma^{j+}}\frac{f^{j}\left(H_{N}^{j}\left(z,y\right)\right)}{z\left(z-x\right)^{2}}\dd z,\ \ (x,y)\in\mathcal{V}_{\rho}^{j}.
\]
We estimate the new integrals similarly to the beginning using 
\[
\left|x^{2}\frac{f\left(H\left(z,y\right)\right)}{z(z-x)^{2}}\dd z\right|\leq\rho A\norm f{}'e^{\norm N{}}q\left(x\right)^{2}\left|z^{-\mu-1}\exp\frac{z^{-k}}{k}\right|\left|\dd z\right|
\]
instead of (\ref{eq:maj_integ}). For $x$ close to $\Gamma^{j,+}$(or
to $\Gamma^{j-1,-}$) we have to use Cauchy's formula similarly to
(\ref{eq:cauchy-gamma}), but obtain on the right hand side $x^{2}\frac{\partial}{\partial x}\left(\frac{1}{x}f^{j}\left(H_{N}^{j}\left(x,y\right)\right)\right)$.
This requires an estimate of 
\[
x\frac{\partial}{\partial x}\left(f^{j}\left(H_{N}^{j}\left(x,y,\right)\right)\right)=\frac{\dd{f^{j}}}{\dd h}\left(H_{N}^{j}\left(x,y\right)\right)H_{N}^{j}\left(x,y\right)\left(-x^{-k}-\mu+x\frac{\partial N}{\partial x}\right)
\]
 on $V^{j}$. This is done analogously to the proof of (5)(a) and
yields the desired result.

For (3), we use instead of (\ref{eq:maj_integ}) 
\[
\left|\frac{f\left(H\left(z,y\right)\right)}{z(z-x)}\dd z\right|\leq\norm f{}'\left|H\left(z,y\right)\frac{\dd z}{z(z-x)}\right|\leq\rho A\norm f{}'e^{\norm N{}}\tilde{q}\left(x\right)\left|z^{-\mu-2}\exp\frac{z^{-k}}{k}\right|\left|\dd z\right|
\]
where $\tilde{q}\left(x\right)=\sup\left\{ \frac{|z|}{|z-x|}\mid z\in\Gamma\right\} .$
This yields that $\left|\Sigma^{j}\left(x,y\right)\right|\leq\left|x\right|\rho\tilde{K}\norm f{}'e^{\norm N{}}$
with some constant $\tilde{K}$ determined in a way analogous to $K$
and hence (3).

Point (4) is a consequence of the fact that the consecutive differences
of components of $\tilde{\Sigma}$ and $\Sigma$ agree hence 
\begin{eqnarray*}
\tilde{\Sigma}^{j}\left(x,y\right)-\Sigma^{j}\left(x,y\right) & =: & \delta\left(x,y\right)
\end{eqnarray*}
defines a bounded, holomorphic function on $\ww C_{\neq0}\times\rho\ww D$.
Riemann's Theorem on removable singularities tells us that $\delta$
can be extended holomorphically to a bounded function on $\ww C\times\rho\ww D$,
which must be a function of $y$ only according to Liouville's theorem. 
\end{proof}

\subsection{Construction of a vector field with given sectorial transition maps }

~

Here we want to find a vector field $X_{R}$ with prescribed transition
maps between sectorial first-integrals. This construction is the core
of the proof for the orbital normal form reduction.
\begin{prop}
\label{prop:realization_transition_maps}Let an admissible collection
$\Delta$ and a collection $\varphi\in\mathcal{B}'\left(\Delta\right)$
be given. Then there exists $\left(\rho,N\right)$ adapted to $\Delta$
such that 
\begin{enumerate}
\item ~
\begin{eqnarray*}
H_{N}^{j+1} & = & H_{N}^{j}\exp\left(\frac{2\ii\pi\mu}{k}+\varphi^{j}\left(H_{N}^{j}\right)\right)\,,
\end{eqnarray*}

\item ~
\begin{eqnarray*}
\norm{y\frac{\partial N}{\partial y}}{} & < & 1\,,\norm{x\frac{\partial N}{\partial x}}{}<1.
\end{eqnarray*}

\end{enumerate}
\end{prop}
\begin{rem}
A value for $\rho$ is given in the proof.
\end{rem}
This proposition relies on a general convergence result for sequences
in the space $\mathcal{B}\left(\Delta\right)$, which we have not
been able to find in the literature but should come in handy in many
problems where the most direct approach is formal.
\begin{lem}
\label{lem:weak_bounded_conv}Let $\Delta$ be a domain in $\ww C^{m}$
and consider a bounded sequence $\left(f_{p}\right)_{p\in\ww N}$
of $\mathcal{B}\left(\Delta\right)$ satisfying the additional property
that there exists some point $\mathbf{z}_{0}\in\Delta$ such that
the corresponding sequence of Taylor series $\left(T_{p}\right)_{p\in\ww N}$
at $\mathbf{z}_{0}$ is convergent in $\frml{\mathbf{z}-\mathbf{z}_{0}}$
equipped with the projective topology. Then $\left(f_{p}\right)_{p}$
converges uniformly on compact sets of $\Delta$ towards some $f_{\infty}\in\mathcal{B}\left(\Delta\right)$.\end{lem}
\begin{rem}
~
\begin{enumerate}
\item The convergence of the sequence of Taylor series $\left(T_{p}\right)_{p\in\ww N}=\left(\sum_{\mathbf{n}\geq0}a_{\mathbf{n}}^{\left(p\right)}\left(\mathbf{z}-\mathbf{z}_{0}\right)^{\mathbf{n}}\right)_{p\in\ww N}$
for the projective topology is equivalent to that of each sequence
$\left(a_{\mathbf{n}}^{\left(p\right)}\right)_{p\in\ww N}$ in $\ww C$.
This is particularly the case when $\left(T_{p}\right)_{p\in\ww N}$
converges for the Krull topology%
\footnote{The one based on the ideals generated by $\left(\mathbf{z}-\mathbf{z}_{0}\right)^{\mathbf{n}}$,
$\mathbf{n}\in\ww N^{m}$.%
} 
\item The convergence might not be uniform on $\Delta$: as an example take
$\Delta:=\left\{ z\in\ww C\,:\,\left|z\right|<1\right\} $ and $f_{p}\left(z\right):=z^{p}$.
\end{enumerate}
\end{rem}
\begin{proof}
Let $\mathcal{O}\left(\Delta\right)$ denote the space of functions
holomorphic on $\Delta$, equipped with the topology of uniform convergence
on compact subsets of $\Delta$, which is a Montel space. If the sequence
$\left(f_{p}\right)_{p}$ is bounded in $\mathcal{B}\left(\Delta\right)$
it is also bounded in $\mathcal{O}\left(\Delta\right)$. Consequently,
there exists a convergent subsequence $\left(f_{p_{j}}\right)_{j\in\ww N}$
in $\mathcal{O}\left(\Delta\right)$; call $f_{\infty}$ its limiting
value. Now Cauchy's integral formula and the uniform convergence $\left(f_{p_{j}}\right)_{j}\to f_{\infty}$
on a small compact polydisc around $\mathbf{z}_{0}$ imply that the
Taylor series of $f_{\infty}$ at $\mathbf{z}_{0}$ coincides with
the limiting power series $\lim_{p}T_{p}$. This argument, together
with the identity theorem on the connected open set $\Delta$, is
sufficient to prove that any other subsequence of $\left(f_{p}\right)_{p}$
converges toward the same function $f_{\infty}$. This implies the
convergence $\left(f_{p}\right)_{p}\to f_{\infty}$ in $\mathcal{O}\left(\Delta\right)$.
The boundedness of the sequence $\left(f_{p}\right)_{p}$in $\mathcal{B}\left(\Delta\right)$
implies that $f_{\infty}$ is bounded by the same constant on each
compact subset of $\Delta$, \emph{i.e.} it belongs to $\mathcal{B}\left(\Delta\right)$.
\end{proof}
We now give a proof of Proposition~\ref{prop:realization_transition_maps}.
\begin{proof}
We recursively define the sequence $\left(N_{n}\right)_{n\in\ww N}$:
starting with $N_{0}:=0$ we put
\begin{eqnarray}
N_{n+1} & := & \Sigma\left(N_{n},\varphi\right),\ n\geq0\label{eq:recursion_N}
\end{eqnarray}
using the refined Cauchy-Heine transform of Definition~\ref{def_RCH}
and Theorem~\ref{thm:Cauchy-Heine}. Then we show it converges in
a convenient Banach space. For the sake of clarity we omit the superscript
$j$ whenever not confusing, and write $H_{n}$ instead of $H_{N_{n}}$. 

We can assume that all $\varphi^{j}$ are holomorphic and have bounded
derivatives on some disc $\eta\ww D\subset\bigcap_{j\in\zsk}\Delta_{j}$.
Then we choose
\begin{eqnarray*}
\rho & \leq & \frac{\eta}{M}e^{-\frac{\eta}{M}K\norm{\varphi'}{\eta\ww D}}
\end{eqnarray*}
where 
\begin{eqnarray*}
M=M\left(k,\mu,\nu,\beta\right) & :=e^{2\pi\left|\mu\right|} & \sup_{z\in V^{j,s}}\left|z^{-\mu}\exp\frac{z^{-k}}{k}\right|
\end{eqnarray*}
and $K$ is the constant appearing in Theorem~\ref{thm:Cauchy-Heine}.
We need to ensure that 
\begin{eqnarray}
\left(\forall n\in\ww N\right)\,\left(\forall y\in\rho\ww D\right)\,\left(\forall z\in V^{j,s}\right) & \,\, & \left|H_{n}\left(z,y\right)\right|\leq\eta\,.\label{eq:H_in_disk}
\end{eqnarray}
By construction of $H_{n}$ we have for $\left(z,y\right)\in\mathcal{V}_{\rho}^{j}$
\begin{eqnarray*}
\left|H_{n}\left(z,y\right)\right| & \leq & \rho e^{2\pi\left|\mu\right|}\left|z^{-\mu}\exp\frac{z^{-k}}{k}\right|e^{\norm{N_{n}}{}}\,.
\end{eqnarray*}
Therefore if for some $n\in\ww N$ we have$\norm{N_{n}}{}\leq\frac{2\pi}{M}\eta K\norm{\varphi'}{\eta\ww D}$
then we first find that
\begin{eqnarray*}
\left|H_{n}\left(z,y\right)\right| & \leq & \rho M\exp\left(\frac{\eta}{M}K\norm{\varphi'}{\eta\ww D}\right)=\eta,
\end{eqnarray*}
\emph{i.e.} $\left(\rho,N_{n}\right)$ is adapted to $\Delta$ and
then, using Theorem~\ref{thm:Cauchy-Heine}~(5)(a), we obtain
\begin{eqnarray*}
\norm{N_{n+1}}{} & \leq & \rho K\norm{\varphi}{\eta\ww D}'\exp\norm{N_{n}}{}\\
 & \leq & \rho K\norm{\varphi'}{\eta\ww D}\exp\norm{N_{n}}{}\\
 & \leq & \frac{\eta}{M}K\norm{\varphi'}{\eta\ww D}\,.
\end{eqnarray*}
 These estimates show by induction on $n$ that, with the above choice
of $\rho$, the relation~(\ref{eq:recursion_N}) defines a bounded
sequence $\left(N_{n}\right)_{n}\subset\mathcal{B}\left(\mathcal{V}_{\rho}^{j}\right)$.
It is then sufficient to show that the sequence $\left(N_{n}\right)_{n}$
converges for the Krull topology on $\mathcal{B}\left(V^{j}\right)\left[\left[y\right]\right]$
to obtain its convergence towards an element $N$ of the Banach space
$\mathcal{B}\left(\mathcal{V}_{\rho}^{j}\right)$ (use Lemma~\ref{lem:weak_bounded_conv}).
By construction this limit is a fixed point of the operator $N\mapsto\Sigma\left(N,\varphi\right)$
(as it is continuous for the compact uniform convergence) and therefore
$N^{j+1}-N^{j}=\varphi^{j}\left(H_{N}^{j}\right)$, according to Theorem~\ref{thm:Cauchy-Heine}.
As an immediate consequence we obtain
\begin{eqnarray*}
\frac{H_{N}^{j+1}}{H_{N}^{j}} & = & \exp\left(\frac{2\ii\pi\mu}{k}+N^{j+1}-N^{j}\right)=\exp\left(\frac{2\ii\pi\mu}{k}+\varphi^{j}\left(H_{N}^{j}\right)\right)
\end{eqnarray*}
and thus we proved (1). 

Now the estimate~(5)(b) in Theorem~\ref{thm:Cauchy-Heine} implies,
for all $n\in\ww N$,
\begin{eqnarray*}
\frac{\norm{y\frac{\partial N_{n+1}}{\partial y}}{}}{1+\norm{y\frac{\partial N_{n}}{\partial y}}{}} & \leq & \rho K\norm{\varphi'}{}\exp\norm{N_{n}}{}\leq\rho K\norm{\varphi'}{}\exp\left(\frac{\eta}{M}K\norm{\varphi'}{}\right)\,.
\end{eqnarray*}
 If we choose $\rho$ so small that also $\rho K\norm{\varphi'}{}\exp\left(\frac{\eta}{M}K\norm{\varphi'}{}\right)=:\alpha<\frac{1}{2}$
then we have shown that $\norm{y\frac{\partial N_{n+1}}{\partial y}}{}\leq\alpha\left(1+\norm{y\frac{\partial N_{n}}{\partial y}}{}\right)$
for all $n$ and this implies that the limit $N$ satisfies $\norm{y\frac{\partial N}{\partial y}}{}\leq\frac{\alpha}{1-\alpha}<1.$
As the estimate for $\norm{x\frac{\partial N}{\partial x}}{}$ follows
in the same way, this establishes (2).

To complete the proof let us show by recursion on $n$ that $N_{n+1}^{j}-N_{n}^{j}=O\left(y^{n+1}\right)$
and hence that the sequence is convergent for the Krull topology.
By construction this is true for $n=0$. Now let us assume the property
is true for some $n$. Since
\begin{eqnarray*}
H_{n+1} & = & H_{n}\exp\left(N_{n+1}-N_{n}\right)=H_{n}\left(1+O\left(y^{n+1}\right)\right)
\end{eqnarray*}
and since $H_{n}=O\left(y\right)$, we find
\begin{eqnarray*}
\varphi\left(H_{n+1}\right) & = & \varphi\left(H_{n}\right)+O\left(y^{n+2}\right)\,.
\end{eqnarray*}
As the integral defining $\Sigma\left(N,\varphi\right)$ is $\ww C\left\{ y\right\} $-linear,
the result follows.\end{proof}
\begin{cor}
\label{cor:vector_field_prescribed_transitions}Let $\rho>0$ and
$N=\left(N^{j}\right)_{j}\in\mathcal{B}\left(\left(\mathcal{V}_{\rho}^{j}\right)_{j}\right)$
be given by Proposition~\ref{prop:realization_transition_maps}.
Then:
\begin{enumerate}
\item the vector fields 
\begin{eqnarray*}
X^{j} & := & X_{0}-y\frac{X_{0}\cdot N^{j}}{1+y\frac{\partial N^{j}}{\partial y}}\pp y
\end{eqnarray*}
are holomorphic on $\mathcal{V}_{\rho}^{j}$ and admit $H_{N}^{j}$
as first integrals, 
\item these vector fields $X^{j}$, for $j\in\zsk$, are the restrictions
to the sectors $\mathcal{V}_{\rho}^{j}$ of a vector field $X=X_{R}$
holomorphic on $\mathcal{V}_{\rho}$ with 
\begin{eqnarray*}
R & \in & y\mathcal{O}\left(\ww C\right)\left\{ y\right\} \,.
\end{eqnarray*}

\end{enumerate}
\end{cor}
\begin{proof}
Define 
\begin{eqnarray*}
\tilde{R}^{j} & := & -\frac{X_{0}\cdot N^{j}}{1+y\frac{\partial N^{j}}{\partial y}}
\end{eqnarray*}
so that $X^{j}=X_{0}+y\tilde{R}^{j}\pp y$. Since $\norm{y\frac{\partial N}{\partial y}}{}<1$
and $\norm{x\frac{\partial N}{\partial x}}{}<1$, we indeed have $\tilde{R}^{j}\in\mathcal{B}\left(\mathcal{V}_{\rho}^{j}\right)$.
Because of Riemann's Theorem on removable singularities, each $\tilde{R}^{j}$
is the restriction of a function $\tilde{R}\in y\mathcal{O}\left(\ww C\right)\left\{ y\right\} $
if, and only if, $\tilde{R}^{j}=\tilde{R}^{j+1}$ on $\mathcal{V}_{\rho}^{j,s}$
for all $j$. This condition is satisfied because of point (1) of
Proposition \ref{prop:realization_transition_maps}. Indeed on the
one hand we have 
\begin{eqnarray*}
X^{j}\cdot H_{N}^{j+1} & = & X^{j}\cdot\left(H_{N}^{j}\exp\left(\frac{2i\pi\mu}{k}+\varphi\left(H_{N}^{j}\right)\right)\right)=0,
\end{eqnarray*}
because $X_{j}\cdot H_{N}^{j}=0$, on the other hand a short calculation
shows that
\begin{eqnarray*}
X^{j}\cdot H_{N}^{j+1} & = & H_{N}^{j+1}\left(X_{0}\cdot N^{j+1}+\left(1+y\frac{\partial N^{j+1}}{\partial y}\right)\tilde{R}^{j}\right)\,.
\end{eqnarray*}
Hence $X_{0}\cdot N^{j+1}+\left(1+y\frac{\partial N^{j+1}}{\partial y}\right)\tilde{R}^{j}=0$
and thus $\tilde{R}^{j}=\tilde{R}^{j+1}$.

Since all $N^{j}\left(x,y\right)$ tend to $0$ as $V^{j}\ni x\to0$
(uniformly for small $y$) we conclude that $\tilde{R}=xR$ with some
$R\in y\mathcal{O}\left(\ww C\right)\left\{ y\right\} $.\end{proof}
\begin{rem}
\label{rem_orbital_realization}~
\begin{enumerate}
\item The final formula is 
\begin{eqnarray*}
R & := & -\frac{X_{0}\cdot N^{j}}{x\left(1+y\frac{\partial N^{j}}{\partial y}\right)}\,,
\end{eqnarray*}
which does not depend on $j$ as stated in the above corollary.
\item $X$ is simply obtained by performing the change of variables $\left(x,y\right)\mapsto\left(x,y\exp N^{j}\left(x,y\right)\right)$
in $X_{0}$. Thus $H_{0}$ is naturally transformed into $H_{N^{j}}$.
Hence the relations
\begin{eqnarray*}
X\cdot N^{j} & = & -xR
\end{eqnarray*}
hold on the sectors and, by the definition of $\per R$ in subsection
1.2.4, we obtain
\begin{eqnarray*}
\left(\varphi^{j}\right)_{j\in\zsk} & = & \per R\left(-xR\right)\,.
\end{eqnarray*}

\item Point~(1) of Proposition~\ref{prop:realization_transition_maps}
states precisely that the Martinet-Ramis modulus of $X$ is 
\begin{eqnarray*}
\psi^{j}\left(h\right) & = & h\exp\left(\frac{2\ii\pi\mu}{k}+\varphi^{j}\left(h\right)\right)\,.
\end{eqnarray*}

\end{enumerate}
\end{rem}
We want to show that $R$ belong to $\polg=y\pol x{<k}\left\{ y\right\} $
but we have only proved so far that $R\in y\mathcal{O}\left(\ww C\right)\left\{ y\right\} $.
We complete the construction of our analytic orbital normal form by
proving that claim.
\begin{lem}
\label{lem:perturb_is_poly}The function $R$ of Corollary \ref{cor:vector_field_prescribed_transitions}
satisfies $R\in\polg$.\end{lem}
\begin{proof}
By Proposition~\ref{prop:realization_transition_maps}, we have
\begin{eqnarray*}
\norm{x\frac{\partial N}{\partial x}}{}<1 & \,\,\,\,\,\mbox{and}\,\,\,\,\, & \norm{y\frac{\partial N}{\partial y}}{}<1\,.
\end{eqnarray*}
and thus 
\begin{eqnarray*}
\frac{1}{\norm{1+y\frac{\partial N}{\partial y}}{}} & \leq & \frac{1}{1-\norm{y\frac{\partial N}{\partial y}}{}}<\infty\,.
\end{eqnarray*}
Since
\begin{eqnarray*}
\left|\left(X_{0}\cdot N^{j}\right)\left(x,y\right)\right| & \leq & \left|x\right|^{k}\norm{x\frac{\partial N}{\partial x}}{}+\left(1+\left|\mu x^{k}\right|\right)\norm{y\frac{\partial N}{\partial y}}{}
\end{eqnarray*}
for $\left(x,y\right)\in\mathcal{V}_{\rho}$ , $j\in\zsk,$ we conclude
from the definition of $R$ that
\begin{eqnarray*}
xR & = & O\left(x^{k}\right)\,.
\end{eqnarray*}
As $x\mapsto R\left(x,y\right)$ is an entire function for every fixed
$\left|y\right|<\rho$, it must be a polynomial of degree at most
$k-1$. 
\end{proof}

\subsection{Uniqueness}

~

In order to complete the proof of the orbital part of the Main Theorem
we only need to address the uniqueness clause.
\begin{prop}
\label{pro:uniqueness}Let an orbital formal class be fixed. Two vector
fields $X_{R}$ and $X_{\tilde{R}}$ with $R,\,\tilde{R}\,\in\polg$
are analytically orbitally conjugate by some $\Psi$ in a neighborhood
of $0\in\ww C^{2}$ if, and only if, there exists $\left(\theta,c\right)\in{\tt Aut}_{k}$
such that
\begin{eqnarray*}
\tilde{R}\left(x,y\right) & = & R\left(e^{\nf{2\ii\pi\theta}k}x,cy\right)\,.
\end{eqnarray*}
In that case there exists $T\in\germ{x,y}$ such that
\begin{eqnarray*}
\Psi\left(x,y\right) & = & \left(\flow{X_{R}}T{}\right)\circ\left(e^{\nf{2\ii\pi\theta}k}x,cy\right)\,,
\end{eqnarray*}
where $\flow{X_{R}}t{}$ is the flow of $X_{R}$ at time $t$.\end{prop}
\begin{proof}
We use a fact proved later in Corollary~\ref{cor:orbit_invertible}:
the map $R\in\polg\mapsto\left(\varphi_{R}^{j}\right)_{j\in\zsk}$,
sending $R$ to the canonical orbital modulus $\per R\left(-xR\right)$
of $X_{R}$, is one-to-one. According to Martinet-Ramis' theorem there
must exist $\left(\theta,c\right)\in{\tt Aut}_{k}$ such that $\varphi_{R}^{j+\theta}\left(ch\right)=\varphi_{\tilde{R}}^{j}\left(h\right)$.
Up to right-composition of $\Psi$ by $\left(x,y\right)\mapsto\left(e^{-\nf{2\ii\pi\theta}k}x,c^{-1}y\right)$
we may therefore assume that $\Psi$ is tangent to the identity and
$\varphi_{R}^{j}=\varphi_{\tilde{R}}^{j}$, so that $R=\tilde{R}$.
We are left with studying the tangent-to-the-identity symmetries of
the foliation induced by $X_{R}$. We have $\Psi^{*}X_{R}=UX_{R}$
for some holomorphic unit $U$ and there exists a holomorphic $T$
such that $T\left(0,0\right)=0$ and $X_{R}\cdot T=\frac{1}{U}-1$
(see Section~\ref{sub:tempo_modulus}). Now up to composition of
$\Psi$ by the inverse of $\flow{X_{R}}T{}$, we are left with studying
the tangent-to-the-identity symmetries of the vector field $X_{R}$.
This can be carried out on a formal level, thereby for $X_{0}$: it
is easy to show that such a formal symmetry of $X_{0}$ must be of
the form $\flow{X_{0}}t{}$ for some $t\in\ww C$, which ends the
proof.
\end{proof}

\section{\label{sec:tempo-NF}The natural section of the period operator and
temporal normal forms}

We start with some vector field $X_{R}$ constructed in the previous
section with prescribed orbital modulus, holomorphic on the domain
$\mathcal{V}_{\rho}=\ww C\times\rho\ww D$ for some well-chosen $\rho>0$,
as described in Proposition~\ref{prop:realization_transition_maps}.
Consider the admissible collection $\Delta$ defined by $\Delta^{j}:=H_{N}\left(\mathcal{V}_{\rho}^{j}\right)$,
whose size shrinks as $\rho$ goes to $0$. We refer to Section~\ref{sub:tempo_modulus}
for the construction of the period operator $\per R$, and the justification
that the Main~Theorem is proved once we have established the following
result.
\begin{prop}
\label{prop:natural_section}Let a collection $f=\left(f^{j}\right)\in\mathcal{B}\left(\Delta\right)$
be given. Then there exists a unique $\sec R\left(f\right)\in x\polg$
such that
\begin{eqnarray*}
\per R\circ\sec R\left(f\right) & = & f\,.
\end{eqnarray*}

\end{prop}
Notice that in addition to the special form of $\sec R$ this proposition
ensures a control on the domain of definition of the normal form and,
more generally, on that of the sectorial solutions to a cohomological
equation.
\begin{proof}
Following Theorem~\ref{thm:Cauchy-Heine} and Definition~\ref{def_RCH}
we obtain sectorial functions $\left(\Sigma^{j}\right)_{j}:=\Sigma\left(N,f\right)$
such that 
\begin{eqnarray*}
\Sigma^{j+1}-\Sigma^{j} & = & f^{j}\circ H_{N}^{j}\,.
\end{eqnarray*}
Each function $\Sigma^{j}$ is holomorphic and bounded on $\mathcal{V}_{\rho}^{j}$
and tends to 0 as $V^{j}\ni x\to0$, uniformly on $\rho\ww D$. Define
now
\begin{eqnarray*}
g & := & X_{R}\cdot\Sigma^{j}\,,
\end{eqnarray*}
which does not depend on $j$ (because $H_{N}^{j}$ is a first integral
of $X_{R}$) and tends to 0 as $x\to0$. Therefore it can be extended
to a function holomorphic on $\mathcal{V}_{\rho}$ by Riemann's Theorem
on removable singularities. As in the proof of Lemma~\ref{lem:perturb_is_poly}
it follows that $g\in x\pol x{<k}\left\{ y\right\} $. 
\end{proof}
By extending the arguments of the proof of Proposition~\ref{pro:uniqueness}
we deduce easily the
\begin{cor}
\label{cor:uniqueness}Let an orbital formal class be fixed. Two vector
fields $U_{G}X_{R}$ and $U_{\tilde{G}}X_{\tilde{R}}$ with $R,\,\tilde{R},\, G,\,\tilde{G}\,\in\polg$,
having respective formal temporal moduli $P$ and $\tilde{P}$, are
analytically conjugate by some $\Psi$ in a neighborhood of $0\in\ww C^{2}$
if, and only if, there exists $\left(\theta,c\right)\in{\tt Aut}_{k}$
such that $\left(P,R,G\right)$ is conjugate to $\left(\tilde{P},\tilde{G},\tilde{R}\right)$
by the right-composition by $\left(x,y\right)\mapsto\left(e^{\nf{2\ii\pi\theta}k}x,cy\right)$.
In that case there exists $t\in\ww C$ such that
\begin{eqnarray*}
\Psi\left(x,y\right) & = & \flow{X_{R}}t{\left(e^{\nf{2\ii\pi\theta}k}x,cy\right)}\,.
\end{eqnarray*}

\end{cor}

\section{\label{sec:Explicit}Explicit realization and algorithms}

In the previous sections, we have already discussed the existence,
uniqueness and convergence of the normal forms and more generally
of the natural section of the period operator. Here, we are only concerned
with the actual computation, numerical or symbolic, of these objects
and no longer think about convergence. 

We present as precisely as possible the steps needed to compute explicitly,
trying to do as much symbolic computations as possible. Nevertheless,
we must allow iterated integrals of some class of transcendental functions
as elementary building blocks.

\subsection{\label{sub:symbolic_convergent}Symbolic approach for a convergent
vector field}

~

In this section, we use the sectors $V^{j}$ of all $x$ with $\left|\arg x-j\frac{2\pi}{k}\right|<\frac{\pi}{k}+\beta,\ \left|x\right|<r$,
for sufficiently small $r>0.$ Unless otherwise stated, $R$ can be
any element of $y\ww C\left\{ x,y\right\} $. 

In order to compute the period $\mathcal{T}_{R}\left(x^{m}y^{n}\right)\left(h\right)=\left(\mathcal{T}_{R}^{j}\left(x^{m}y^{n}\right)\left(h\right)\right)_{j\in\ww Z/k\ww Z}$,
we have to integrate the differential form $x^{m-k-1}y^{n}\dd x$
over the asymptotic cycle $\gamma^{j,s}\left(h\right)$ included in
the sectorial leaf $\left\{ H_{N}^{j}=h\right\} $ (see Definition~\ref{def_period}),
where $H_{N}^{j}$ denotes the sectorial first integrals associated
to $X_{R}$ (see Definition~\ref{def_first-integral}). 

We are particularly interested in inverting the relations
\begin{eqnarray}
\per R^{j}\left(\sum_{n\geq1}G_{n}\left(x\right)x^{\sigma n+1}y^{n}\right)\left(h\right) & = & \sum_{\ell\geq1}f_{\ell}^{j}h^{\ell},\ j\in\nf{\ww Z}{k\ww Z},\label{eq:system}
\end{eqnarray}
with $G_{n}\left(x\right)=\sum_{m=0}^{k-1}G_{m,n}x^{m}\in\ww C[x]_{<k}$,
being given a $k$-tuple $\left(f^{j}\right)_{j\in\nf{\ww Z}{k\ww Z}}$
of formal power series $f^{j}\left(h\right)=\sum_{\ell\geq1}f_{\ell}^{j}h^{\ell}$.
It turns out that the corresponding system, expressing the infinite
vector $\left(f_{\ell}^{j}\right)_{\ell,j}$ in terms of the vector
$\left(G_{m,n}\right)_{m,n}$, is an invertible block-triangular system,
if $\sigma+\mu\not\in\ww R_{\leq0}$. This condition will be assumed
throughout the section. 

This section is devoted to proving the 
\begin{prop}
\label{pro:period_triangular}Let $R\left(x,y\right):=\sum_{n>0}R_{n}\left(x\right)y^{n}\in y\,\germ{x,y}$.
The coefficients of the periods $\mathcal{T}_{R}^{j}\left(x^{m}y^{n}\right)\left(h\right)=\sum_{\ell\geq0}c_{m,n,\ell}^{j}h^{\ell},\ j\in\nf{\ww Z}{k\ww Z},\, m,n\in\ww N$
satisfy the following properties.\end{prop}
\begin{enumerate}
\item $c_{m,n,\ell}^{j}=0$, if $\ell<n$ and $c_{m,n,n}^{j}$ is independent
of $R.$
\item For $\ell>n,$ the coefficients $c_{m,n,\ell}^{j}$ depend only on
$R_{1},\ldots,R_{\ell-n}$ and vanish when $R=0$. The $k\times k$
matrices $D_{n}:=\tx{diag}\left(c_{m,n,n}^{0}\,:\, n\sigma+1\leq m\leq n\sigma+k\right)$
and $V:=\left[\exp\left(\nf{2\ii\pi mj}k\right)\right]_{\left(m,j\right)}$
are invertible. The relations~(\ref{eq:system}) are satisfied if,
and only if,
\[
\left[f_{n}^{j}\right]_{j\in\nf{\ww Z}{k\ww Z}}=VD_{n}\left[G_{m,n}\right]_{m<k}+\left[\sum_{1\leq a<n}\sum_{0\leq m<k}G_{m,a}c_{n\sigma+m+1,a,n}^{j}\right]_{j\in\nf{\ww Z}{k\ww Z}}.
\]
If $R\in\polg[x^{\sigma}y]$ and $\ell>n$ then $c_{m,n,\ell}^{j}$
is a polynomial in the $k\left(\ell-n\right)$ variables given by
the coefficients of $R_{1},\cdots,R_{\ell-n}$. Its coefficients can
be symbolically computed.\end{enumerate}
\begin{rem}
~
\begin{enumerate}
\item As $c_{m,n,n}^{j}$ do not depend on $R$, their values can be computed
when $R=0$ (see~\cite{Eliza,Tey-EqH}). We recall this result in
the next subsection. 
\item The coefficients mentioned in~(2) above can be computed as $\int_{\eta^{j}}x^{m+n\mu}e^{-\frac{n}{k}x^{-k}}Q\left(x\right)\dd x$,
where $Q$ is a polynomial of some iterated integrals involving only
powers of $x$ and exponentials and where $\eta^{j}$ is the projection
of some asymptotic cycle $\gamma^{j,s}\left(h\right)$ onto the $x$-plane.
\end{enumerate}
\end{rem}
Before giving the proof, we state two direct consequences of this
proposition. The first statement has been used in the proof of Proposition
\ref{pro:uniqueness}.
\begin{cor}
\label{cor:orbit_invertible}Finding $R(x,u)=\sum_{n\geq1}R_{n}\left(x\right)u^{n},\ R_{n}\left(x\right)=\sum_{m=0}^{k-1}R_{m,n}x^{m}\in\ww C\left[x\right]_{<k}$,
such that $X_{R}$ realizes a given orbital invariant $\varphi\in\left(h\ww C\left\{ h\right\} \right)^{k}$
means solving
\begin{eqnarray*}
\per R\left(\sum_{n\geq1}R_{n}\left(x\right)x^{\sigma n+1}y^{n}\right) & = & -\varphi\,.
\end{eqnarray*}
This system is non-linear but again «block-triangular» and formally
invertible. More precisely, the equations determining the vector $\left(R_{m,n}\right)_{m<k}$
are 
\[
-\left[f_{n}^{j}\right]_{j\in\nf{\ww Z}{k\ww Z}}=VD_{n}\left(R_{m,n}\right)_{m<k}+K_{n}\left(R_{1},...,R_{n-1}\right),
\]
where $K_{n}\left(R_{1},...R_{n-1}\right)$ denotes the (vector) coefficient
of $h^{n}$of $\per{\tilde{R}}\left(x\tilde{R}\left(x,y\right)\right)\left(h\right),$
$\tilde{R}\left(x,y\right)=\sum_{\ell=1}^{n-1}R_{\ell}\left(x\right)x^{\sigma\ell}y^{\ell}$
and hence depends only upon the previously determined $R_{\ell}$.
\end{cor}
In the next section, the subsequent corollary will enable our numerical
computations.
\begin{cor}
\label{cor:step-by-step}Let $R=\sum_{p\geq0}R_{p}\left(x\right)y^{p}\in\germ{x,y}$
and $n,\, m,\, d\in\ww N$ be given. We denote by $\tilde{R}_{d}$
the truncated function $\sum_{j\leq d}R_{j}\left(x\right)y^{j}$.
Then
\begin{eqnarray*}
\mathcal{T}_{R}\left(x^{m}y^{n}\right)\left(h\right) & = & \mathcal{T}_{\tilde{R}_{d}}\left(x^{m}y^{n}\right)\left(h\right)+o\left(h^{n+d}\right)\,.
\end{eqnarray*}

\end{cor}

\subsubsection{The case of the model}

~

The leaf $\left\{ h=H_{0}^{j}\right\} $ is the graph of the function
given by 
\begin{eqnarray*}
y\left(x\right) & = & hx^{\mu}\exp\left(-j\frac{2\ii\pi\mu}{k}-\frac{1}{kx^{k}}\right)\,.
\end{eqnarray*}
 Notice that, because of the determination of $\arg x$ in $V^{j}$,
the right-hand side of the above relation depends only on the class
of $j$ in $\zsk$ . For convenience we define
\begin{eqnarray*}
\delta & := & \exp\left(\frac{2\ii\pi}{k}\right)\\
E\left(x\right) & := & x^{\mu}\exp\left(-\frac{1}{kx^{k}}\right)\,.
\end{eqnarray*}
Letting $\eta^{j}$ denote the projection of $\gamma^{j,s}\left(h\right)$
on the plane $\left\{ y=0\right\} $ (which does not depend on $h$)
we compute
\begin{eqnarray*}
\per 0^{j}\left(x^{m}y^{n}\right)\left(h\right) & = & e^{-nj\,\nf{2\ii\pi\mu}k}h^{n}\int_{\eta^{j}}x^{m}E\left(x\right)^{n}\frac{\dd x}{x^{k+1}}=\delta^{mj}\frac{2\ii\pi\left(\nf nk\right)^{\frac{m+n\mu}{k}}e^{\ii\pi\left(m+n\mu\right)/k}}{n\Gamma\left(\frac{m+n\mu}{k}\right)}h^{n}\,.
\end{eqnarray*}
The value of the integral is computed using Hankel's integral representation
of $\frac{1}{\Gamma}$. This computation has been performed first
by P.~\noun{Elizarov} in~\cite{Eliza} to compute Gâteaux derivatives
of the orbital modulus along the direction $R\in\ww Cx^{n}y^{m}$.
The coefficient 
\begin{eqnarray}
c_{m,n}^{j} & := & \delta^{mj}\frac{2\ii\pi\left(\nf nk\right)^{\frac{m+n\mu}{k}}e^{\ii\pi\left(m+n\mu\right)/k}}{n\Gamma\left(\frac{m+n\mu}{k}\right)}\label{eq:period-coef-1}\\
 & = & c_{m,n,n}^{j}\nonumber 
\end{eqnarray}
is zero if, and only if, $m+n\mu\in k\ww Z_{\leq0}$. Hence, the condition
$\sigma+\mu\notin\ww R_{\leq0}$ prevents $c_{n\sigma+m,n}^{j}$ from
vanishing (as long as $m$ is nonnegative).

\subsubsection{General case: proof of Proposition~\ref{pro:period_triangular}}

~

Here the leaf $\left\{ h=H_{N}^{j}\right\} $ is the graph of the
function $x\mapsto y\left(x\right)$ given by 
\begin{eqnarray}
y\left(x\right) & = & \Theta^{j}\left(x,hx^{\mu}\exp\left(-j\frac{2\ii\pi\mu}{k}-\frac{1}{kx^{k}}\right)\right)\label{eq:param_y}
\end{eqnarray}
where 
\begin{eqnarray*}
\Theta^{j} & := & \left(\Psi_{O}^{j}\right)^{-1}
\end{eqnarray*}
denotes the inverse of the sectorial normalization introduced in subsection
1.2.2. Let us expand $\Theta^{j}$ with respect to $y$:
\begin{eqnarray*}
\Theta^{j}\left(x,y\right) & = & \left(x,\sum_{\ell\geq0}\theta_{\ell}^{j}\left(x\right)y^{\ell}\right)\qquad\qquad,\qquad\theta_{0}^{j}:=0\,,\,\theta_{1}^{j}:=1\,
\end{eqnarray*}
and denote by $\eta^{j}$ the projection of $\gamma^{j,s}\left(h\right)$
on the plane $\left\{ y=0\right\} $ (which does not depend on $h$
nor on $R$). For given $\xi\in\mathcal{B}\left(V^{j}\right)$ with
$\xi\left(x\right)=O\left(x^{k+1}\right)$, we set 
\begin{eqnarray*}
T_{m}^{j}\left(\xi\right) & := & \int_{\eta^{j}}x^{m}\xi\left(x\right)\frac{\dd x}{x^{k+1}}\,.
\end{eqnarray*}
We have the formula
\begin{eqnarray}
\mathcal{T}_{R}^{j}\left(x^{m}y^{n}\right)\left(h\right) & = & c_{m,n}^{j}h^{n}\label{eq:period_formula}\\
 &  & +\sum_{\ell>n}h^{\ell}e^{-\nf{2\ii\pi\mu j\ell}k}T_{m}^{j}\left(E^{\ell}\sum_{\ell_{1}+\cdots+\ell_{n}=\ell}\prod_{p}\theta_{\ell_{p}}^{j}\right)\,.\nonumber 
\end{eqnarray}
This implies statement (1) of Proposition \ref{pro:period_triangular}
and, together with formula (\ref{eq:period-coef-1}) and the linearity
of $\per R$, also statement (3).

\bigskip{}

It remains to prove (2) and (4). We write
\begin{eqnarray*}
R\left(x,y\right) & = & \sum_{n>0}R_{n}\left(x\right)y^{n}\,\,\,\,,
\end{eqnarray*}
where $R_{n}\in\ww C\left\{ x\right\} $ have a common disk of convergence.
We can explicit the normalizing functions themselves, because each
coefficient of 
\begin{eqnarray*}
N^{j}(x,y) & := & \sum_{n>0}N_{n}^{j}\left(x\right)y^{n}\,\,\,\,,\, N_{n}^{j}\in\mathcal{O}\left(V^{j}\right)
\end{eqnarray*}
is the unique solution bounded on $V^{j}$ of a first-order, linear
and inhomogeneous differential equation we deduce from
\begin{eqnarray*}
X_{R}\cdot N^{j} & = & -xR
\end{eqnarray*}
by identifying the coefficients of $y^{n}$. Thus, we have 
\begin{eqnarray}
N_{n}^{j}\left(x\right) & = & E^{-n}\left(x\right)\int_{\left(0\to x\right)}\Delta_{n}^{j}\left(t\right)E^{n}\left(t\right)t^{-k}dt\label{eq:secto_normalize_solve}
\end{eqnarray}
 where 
\begin{eqnarray*}
\Delta_{n}^{j}\left(x\right) & := & R_{n}\left(x\right)-\sum_{p+q=n}q\, N_{q}^{j}\left(x\right)R_{p}\left(x\right)
\end{eqnarray*}
with $N_{0}^{j}=R_{0}^{j}=0$. Here the integration is done on the
projection $\left(0\to x\right)$ of an asymptotic path (see Figure~\ref{fig:path_asy}).
It follows that $N_{n}^{j}$ depends on $R_{p}$ for $0<p\leq n$.
In the case of $R\in\polg[x^{\sigma}y]$, we can be more precise.
For simplicity, we state this only if $k=1.$
\begin{lem}
If k=1 and $R(x,y)=\sum_{n>0}R_{n}x^{n\sigma}y^{n}\,$ with $R_{n}\in\ww C$,
then there exist <<universal>> functions $\phi_{\bullet}^{n}$,
obtained as sums and products of iterated integrals, such that each
function $N_{n}^{0}$ can be written
\begin{eqnarray*}
N_{n}^{0}\left(x\right) & = & \sum_{\ell=1}^{n}\sum_{\begin{array}{c}
1\leq j_{1}\leq j_{2}\leq\cdots\leq j_{\ell}\\
j_{1}+j_{2}+\cdots+j_{\ell}=n
\end{array}}\left[\phi_{\left(j_{1},\cdots,j_{\ell}\right)}^{n}\left(x\right)\prod_{m=1}^{\ell}R_{j_{m}}\right]\,.
\end{eqnarray*}
The degree of $N_{n}^{0}$ as a polynomial in the variables $R_{1},\, R_{2},\,\cdots,\, R_{n}$
is exactly $n$. 
\end{lem}
\begin{figure}
\begin{centering}
\includegraphics[height=6cm]{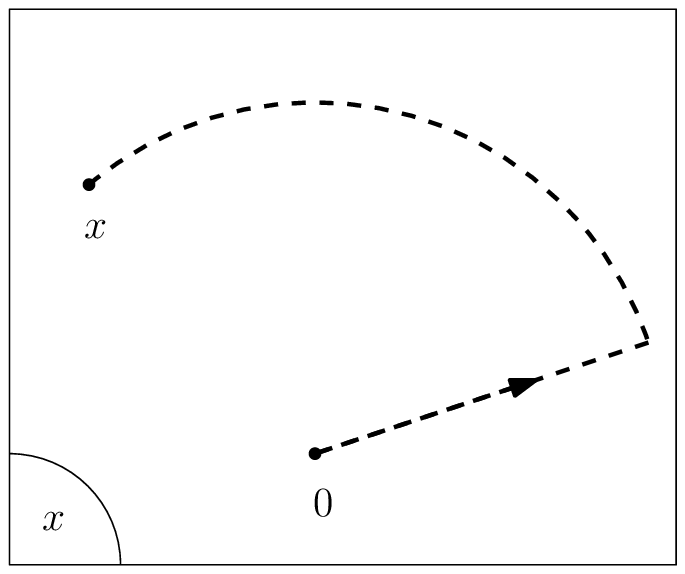}
\par\end{centering}

\caption{\label{fig:path_asy}The path $\left(0\to x\right)$ on which the
integrations are performed to compute $N$.}
\end{figure}

Since
\begin{eqnarray*}
\Theta^{j}\left(x,y\right) & = & \left(x,y+\sum_{n>1}\theta_{n}^{j}\left(x\right)y^{n}\right)\\
 & = & \left(x,y\exp N^{j}\left(x,y\right)\right)^{-1}\,,
\end{eqnarray*}
the properties of the formal inversion imply that $\theta_{n}^{j}$
is a polynomial with integer coefficients of $N_{1}^{j},...,N_{n-1}^{j}$,
hence it depends only on $R_{1},...,R_{n-1}$. By (\ref{eq:period_formula}),
$c_{m,n,\ell}^{j}$ only involves $\theta_{1}^{j},...,\theta_{\ell+1-n}^{j}$
and hence depends only on $R_{1},...,R_{\ell-n}$. This proves statement
(2) of the proposition. In the case of $R\in\polg[x^{\sigma}y],\ R=\sum_{n>0}\sum_{m<k}R_{m,n}x^{m+\sigma n}y^{n}\,\,\,\,,$
this also proves statement (4), because $T_{m}^{j}\left(E^{\ell}\sum_{\ell_{1}+\cdots+\ell_{n}=\ell}\prod_{p}\theta_{\ell_{p}}\right)$
can then be expressed as a polynomial in the $R_{m,s},m<k,\ s\leq\ell-n$
the coefficients of which are of the desired form.

\subsection{\label{sub:Computing}Computing the modulus and normal form: proof
of the Computation~Theorem}

~

In this section we deal with finding an algorithm to compute numerically
the modulus $\varphi^{j}$ and $f^{j}$, as well as the normal form.
We do not intend to give effective nor specially clever methods, but
only a theoretical mean to actually compute. 

The proof of the Computation Theorem follows from the study conducted
here for vector fields written in Dulac's prepared form, as putting
$Z$ into this form is a computable process (the procedure can be
found in \noun{H.~Dulac}'s memoir~\cite{Dulac2}) once the orbital
formal class $\left(k,\mu\right)$ is known%
\footnote{The integer $k$ is not computable with halting, finite Turing machines
as one must test the equality to zero of diverse Taylor coefficients
of the components of $Z$.%
}. We therefore start from a vector field in the (not necessarily normal)
form 
\begin{eqnarray}
Z\left(x,y\right) & = & U\left(x,y\right)X\left(x,y\right)\nonumber \\
X\left(x,y\right) & = & x^{k+1}\pp x+y\left(1+\mu x^{k}+xR\left(x,y\right)\right)\pp y\label{eq:Dulac}
\end{eqnarray}
where $U\left(0,0\right)\neq0$. The formal orbital modulus is explicit
in this form, and the temporal modulus $P$ simply coincides with
the $k^{\tx{th}}$-jet of $U\left(x,0\right)$.

\subsubsection{The period of a given convergent vector field}

~

We want to compute numerically the power series 
\begin{eqnarray*}
\mathcal{T}_{R}^{j}\left(x^{m}y^{n}\right) & = & \sum_{p\geq n}c_{m,n,p}^{j}h^{p}
\end{eqnarray*}
corresponding to the period operator associated to $X_{R}$.
\begin{itemize}
\item We begin with fixing a family of base points $\left(x_{j}\right)_{j\in\ww Z/k}$
in the saddle-parts $V^{j,s}$, for instance $x_{j}=-re^{2\ii\pi j/k}$
where $r>0$ is sufficiently small so that $X$ is defined, but not
too small in order to avoid numerical instabilities.
\item We compute the values $H^{j}\left(y\right):=H_{N}^{j}\left(x_{j},y\right)$
of the sectorial first integrals by integrating numerically $R\frac{\dd x}{x^{k}}$
along an asymptotic path $\gamma^{j}\left(x_{j},y\right)$. One can
think of a Kutta-Runge method to compute $x\mapsto y\left(x\right)$.
\item We compute the sectorial solutions $F^{j}$ to the equation $X\cdot F=x^{m}y^{n}$
in the same way.
\item Hence by definition 
\begin{eqnarray*}
\mathcal{T}^{j}\left(x^{m}y^{n}\right)\left(H^{j}\left(x_{j},y\right)\right) & = & F^{j+1}\left(x_{j},y\right)-F^{j}\left(x_{j},y\right)=\sum_{p\geq n}c_{m,n,j,p}H^{j}\left(x_{j},y\right)^{p}
\end{eqnarray*}
 is a known function $T^{j}$ of $y$. 
\item We derive from this function the coefficients $c_{m,n,\ell}^{j}$
by applying Cauchy's formula:
\begin{eqnarray*}
c_{m,n,\ell}^{j} & = & \frac{1}{2\ii\pi}\int_{\mathcal{C}}\frac{T^{j}\left(y\right)}{H^{j}\left(y\right)^{\ell+1}}\mbox{d}H^{j}\left(y\right)
\end{eqnarray*}
where $\mathcal{C}$ is a circle in $y$-coordinates. Because $H^{j}$
is a diffeomorphism then $H^{j}\left(\mathcal{C}\right)$ is also
a simple loop with unitary winding number around $\left\{ H^{j}=0\right\} $.
\end{itemize}

\subsubsection{Building the normal form}

~

We only deal with the case $k=1$, the general case being the same
up to solving a Vandermonde system. Because of Corollaries~\ref{cor:orbit_invertible},
\ref{cor:step-by-step} one can compute $R\left(x,y\right)=\sum_{n>0}R_{n}x^{\sigma n}y^{n}$
in much the same way as we did before.
\begin{itemize}
\item We fix some base point $x_{0}$.
\item Given $\varphi\left(h\right)=:\sum_{n>0}\alpha_{n}h^{n}$ we compute
$R_{1}=-\alpha_{1}/c_{\sigma,1}$.
\item For $n\geq2$, if $R_{1},...,R_{n-1}$ are already known, we put $\tilde{R}\left(x,y\right)=\sum_{\ell=1}^{n-1}R_{\ell}\left(x\right)x^{\sigma\ell}y^{\ell}$
and compute the period $\mathcal{T}_{\tilde{R}}\left(x\tilde{R}\mbox{(x,y)}\right)$
using the previous method, in order to obtain its coefficient $d_{n}$
of $h^{n}$. 
\item Then, we compute $R_{n}:=\left(-\alpha_{n}-d_{n}\right)/c_{\sigma n,n}$.In
this way, we obtain numerical values for $R_{n}$ in finite time,
up to any order and with arbitrary precision.
\end{itemize}

\subsection{\label{sub:integrability}Integrability by quadrature}

~

These numerical computations actually yield a numerical criterion
for integrability by quadrature of saddle-node equations or, more
correctly, a numerical test of non-integrability by quadrature for
saddle-node convergent vector-fields. Indeed a result by \noun{M.~Berthier}
and \noun{F.~Touzet~}\cite{BerTouze} states that those foliations
corresponding to first order differential equations which are integrable
by quadrature must have their orbital modulus of the form
\begin{eqnarray*}
\varphi^{j}\left(h\right) & = & \frac{1}{p}\log\left(1-\alpha_{j}h^{p}\right)
\end{eqnarray*}
for some $p\in\ww N$ and some collection $\left(\alpha_{j}\right)_{j\in\ww Z/k}\subset\ww C$.
In~\cite{Tey-ExSN} we already proved that their normal form must
then be 
\begin{eqnarray*}
X & = & x^{k+1}\pp x+\left[y\left(1+\mu x^{k}\right)+y^{p+1}x^{p\sigma+1}R_{p}\left(x\right)\right]\pp y.
\end{eqnarray*}

\subsection{Explicit realization of a holonomy diffeomorphism}

~

The manner \noun{J.~Martinet} and \noun{J.-P.~Ramis} identified
completely the space of invariants (\emph{i.e.} proved the «orbital
modulus mapping» is onto) is geometric. They build an abstract almost-complex
$C^{\infty}$-manifold resembling the suspension of the modulus, and
using Newlander-Niremberg's theorem obtain the complex integrability
of this manifold and show it is biholomorphic to a domain of $\ww C^{2}$.
Although this construction is far from being explicit%
\footnote{Even if Newlander-Niremberg's theorem ultimately relies on some fixed-point
method, it appears difficult to translate the proof into a computable
process (as in Definition~\ref{def_computability}) to derive a particular
representative of a given computable orbital class. %
} it nonetheless answers an important question: 
\begin{thm}
\cite{MaRa-SN}Any germ of a diffeomorphism $\psi\in\mbox{Diff}\left(\ww C,0\right)$
can be realized as the holonomy of some convergent saddle-node foliation
singular at $\left(0,0\right)$.
\end{thm}
Indeed set $\mu:=\frac{1}{2\ii\pi}\log\psi'\left(0\right)$ and take
a vector field $X$, in Dulac's form~(\ref{eq:Dulac}), whose orbital
modulus is precisely $\varphi\,:\, h\mapsto\log\frac{\psi\left(h\right)}{h}-2\ii\pi\mu$.
Then the holonomy $\hol{}$ computed by lifting a generator of $\left\{ y=0\,,\, x\neq0\right\} $
in the foliation through the projection $\left(x,y\right)\mapsto x$
is conjugate to $\psi$ through the first-integral. More precisely,
by taking $x_{*}\in V^{s}$ sufficiently close to $0$ and denoting
by $H_{*}$ the local diffeomorphism $y\mapsto H_{0}^{0}\left(x_{*},y\right)$
one obtains, for all $y$ sufficiently close to $0$:
\begin{eqnarray}
\psi\left(H_{*}\left(y\right)\right) & = & H_{*}\left(\hol{}\left(y\right)\right)\,.\label{eq:conj_holonomy}
\end{eqnarray}
Therefore performing the local changes of coordinates $\left(x,y\right)\mapsto\left(x,H_{*}\left(y\right)\right)$
within $X$ produces a new vector field $Z$ in Dulac's form for which
the holonomy computed above $\left\{ x=x_{*}\right\} $ is precisely
$\psi$.

\noun{R.~Perez-Marco} and \noun{J.-C.~Yoccoz} show in~\cite{MarcYoc}
a result of the same kind, by using a quasi-conformal suspension of
$\psi$ and by solving the $\overline{\partial}$-operator equation
to modify the foliated space, making it a domain of $\ww C^{2}$.
Here again the proof is not explicit.

\bigskip{}

The work conducted here allows one to build a somehow explicit realization
of some germ of a diffeomorphism $\psi$ as the holonomy of a foliation
of a normal form. In particular if $\psi$ is computable then so is
$Z$. Besides it is possible to control quite precisely the domain
on which this diffeomorphism is realized.

\subsection{Numerical results }

~

\subsubsection{First example: modulus of an integrable vector field }

~

This is the numerical computation we did for the orbital modulus $\varphi\left(h\right):=\sum_{n\in\ww N}\alpha_{n}h^{n}$
of 
\begin{eqnarray*}
X_{y} & := & x^{2}\pp x+y\left(1+xy\right)\pp y\,.
\end{eqnarray*}
As this equation is a Bernoulli equation its orbital modulus can be
computed explicitly $\psi\,:\, h\mapsto\frac{h}{1-2\ii\pi h}$. Hence
the expected value of $\mathcal{T}_{R}\left(-xR\right)$ is $\varphi\,:\, h\mapsto\log\left(1-2\ii\pi h\right)$.
This is what was computed, using a Kutta-Runge method of order $4$
with a step of $0,001$ for the sectorial integrals, implemented in
${\tt C++}$. The initial condition is $x_{0}=-5$ and the circle
$\mathcal{C}\,:\, t\in\left[0,1\right]\mapsto0,1\times\exp\left(2\ii\pi t\right)$
has been discretized by $1000$ points. Cauchy integrals were computed
using the rectangle rule, which is potentially the best method when
integrating an analytic and periodic function over a period, and $\mbox{d}H\left(y\right)$
was calculated with a $5$-points centered method (also of order $4$).\\

\hfill{}%
\begin{tabular}{|c|c|c|}
\hline 
$n$ & computed $\alpha_{n}$ & expected $\alpha_{n}$\tabularnewline
\hline 
\hline 
$0$ & $-1\times10^{-17}-\ii\,7\times10^{-18}\begin{array}{c}
\,\\
\,
\end{array}$ & $0$\tabularnewline
\hline 
$1$ & $-4\times10^{-17}-\ii\,6,28318530\begin{array}{c}
\,\\
\,
\end{array}$ & $-2\ii\pi\simeq-\ii\,6,28318530$\tabularnewline
\hline 
$2$ & $19,73920883-\ii\,8\times10^{-9}\begin{array}{c}
\,\\
\,
\end{array}$ & $2\pi^{2}\simeq19.73920880$\tabularnewline
\hline 
$3$ & $-2\times10^{-7}-\ii\,82,68340412\begin{array}{c}
\,\\
\,
\end{array}$ & $\frac{8}{3}\ii\pi^{3}\simeq82,68340448$\tabularnewline
\hline 
4 & $1\times10^{-6}+\ii\,-389,63636503\begin{array}{c}
\,\\
\,
\end{array}$ & $-4\pi^{4}\simeq\ii\,389,63636414$\tabularnewline
\hline 
\end{tabular}\hfill{}

We can see that this method is fast and provides results with an error
of the order of $10^{-10+n}$ for the coefficient $\alpha_{n}$. This
shift in the precision is due to the fact that one must divide by
$H^{n+1}$ in Cauchy's formula and $\left|H\right|$ is of the order
of $0,1$.

\subsubsection{Second example: modulus of a non-integrable vector field}

~

This is the numerical computation we did for the orbital modulus $\varphi\left(h\right):=\sum_{n>0}\alpha_{n}h^{n}$
of
\begin{eqnarray*}
X_{y+y^{2}} & := & x^{2}\pp x+y\left(1+x\left(y+y^{2}\right)\right)\pp y\,.
\end{eqnarray*}
We know from the theory that this equation cannot be integrated by
quadrature. The following result is obtained with the same numerical
parameters as previously, keeping the $10-n$ first digits:

\hfill{}%
\begin{tabular}{|c|c|}
\hline 
$n$ & computed $\alpha_{n}$\tabularnewline
\hline 
\hline 
$1$ & $-2\ii\pi\begin{array}{c}
\,\\
\,
\end{array}$\tabularnewline
\hline 
$2$ & $-19,73920883-\ii\,6,28318531\begin{array}{c}
\,\\
\,
\end{array}$\tabularnewline
\hline 
$3$ & $59,2176264+\ii\,78,3282319\begin{array}{c}
\,\\
\,
\end{array}$\tabularnewline
\hline 
$4$ & $-295,429240+\ii\,447,039460\begin{array}{c}
\,\\
\,
\end{array}$\tabularnewline
\hline 
\end{tabular}\hfill{}

If the equation were integrable by quadrature then its modulus would
be of the form $\varphi\left(h\right)=\log\left(1-\alpha h\right)$
for some $\alpha\in\ww C$, which is not possible (provided, of course,
that the numerical errors are of the same magnitude as in the previous
example).

\subsubsection{Third example: realization of a normal form}

~

Here we compute the first terms of the normal form for the modulus
$\varphi\left(h\right)=h$ with $\mu=0$, which is the simplest non-trivial
example. This is what was computed, using a Kutta-Runge method of
order $4$ with a step of $0,001$ for the sectorial integrals. The
initial condition is $x_{0}=-5$ and the circle $\mathcal{C}\,:\, t\in\left[0,1\right]\mapsto0,01\times\exp\left(2\ii\pi t\right)$
has been discretized by $5000$ points. Cauchy integrals were computed
using the rectangle rule. Only the $14-2n$ first digits were kept.

The modulus of this normal form has been conversely evaluated as $h+\sum_{n=2}^{5}\varepsilon_{n}h^{n}+o\left(h^{5}\right)$
with $\left|\varepsilon_{n}\right|<10^{-9}$.

\hfill{}%
\begin{tabular}{|c|c|}
\hline 
\multicolumn{1}{|c||}{$n$} & computed $R_{n}$\tabularnewline
\hline 
\hline 
$1$ & $\ii\,0,159154943092\begin{array}{c}
\,\\
\,
\end{array}$\tabularnewline
\hline 
$2$ & $-\ii\,0,0397887357\begin{array}{c}
\,\\
\,
\end{array}$\tabularnewline
\hline 
$3$ & $-2,27086\times10^{-3}+\ii\,1,473657\times10^{-2}\begin{array}{c}
\,\\
\,
\end{array}$\tabularnewline
\hline 
$4$ & $2,223\times10^{-3}-\ii\,6,239\times10^{-3}\begin{array}{c}
\,\\
\,
\end{array}$\tabularnewline
\hline 
$5$ & $-1,7\times10^{-3}+\ii\,2,8\times10^{-3}\begin{array}{c}
\,\\
\,
\end{array}$\tabularnewline
\hline 
\end{tabular}\hfill{}

\bibliographystyle{plain}
\bibliography{biblio}

\end{document}